\documentclass[lettersize,journal]{IEEEtran}
\usepackage{amsmath,amssymb,amsfonts}
\usepackage{algorithmic}
\usepackage{algorithm}
\usepackage{array}
\usepackage[caption=false,font=normalsize,labelfont=sf,textfont=sf]{subfig}
\usepackage{textcomp}
\usepackage{stfloats}
\usepackage{url}
\usepackage{verbatim}
\usepackage{graphicx}
\usepackage{xcolor}
\newtheorem{lemma}{\bf Lemma}
\usepackage{cite}
\newtheorem{assumption}{\bf Assumption}
\newtheorem{Remark}{\bf Remark}
\newenvironment{proof}{{\it \bf Proof:}}{\hfill $\blacksquare$\par}
\newtheorem{theorem}{\bf Theorem}
\begin{document}

\title{Efficiency of cooperation incentives in evolutionary population games under payoff-observation errors}

\author{Shengxian~Wang, Chengyu~Yin, Xiaojie~Chen, and  Ming~Cao,~\IEEEmembership{Fellow,~IEEE}
        % <-this % stops a space
\thanks{This research was supported by the National Natural Science Foundation of China (Grant Nos. 62406006 and 62473081).  (Corresponding authors: Xiaojie Chen and Ming Cao ).}
\thanks{S. Wang and C. Yin are with School of Computer and Information, Anhui Normal University, Wuhu 241002, China (e-mail: shengxian.wang@ahnu.edu.cn).}
\thanks{X. Chen is with School of Mathematical Sciences, University of Electronic Science and Technology of China, Chengdu 611731, China (e-mail: xiaojiechen@uestc.edu.cn).}
\thanks{M. Cao is with ENTEG, Faculty of Science and Engineering, University of Groningen, Groningen 9747 AG, The Netherlands (e-mail: m.cao@rug.nl).}}

%{Wang \MakeLowercase{\textit{et al.}}: Efficiency of cooperation incentives in evolutionary population games under payoff-observation errors}


\maketitle

\begin{abstract}
Traditional studies on evolutionary dynamics of cooperation have concentrated on an idealized game setup free of payoff-observation errors.  However, in real-world scenarios, individuals frequently encounter errors when observing the payoffs of their opponents during game interactions, resulting from unintentional mistakes, such as data misrecording, overlooking critical details, and miscalculations. This gap between idealized error-free game models and the error-prone real-world game interactions leads to the lack of insight into the impact of payoff-observation errors  on the evolutionary dynamics of population games, in particular the efficiency of cooperation incentives. In this paper, we construct a research framework for population games with payoff-observation errors, which enables us to investigate the effects of errors on the evolutionary dynamics of cooperation in the evolutionary Prisoner's Dilemma game with combined incentives. To quantify the implementation costs of incentives in the presence of errors, we devise an index function and employ optimal control theory to derive the optimal incentive protocols.  Our theoretical and numerical results reveal that payoff-observation errors can lower the costs compared to error-free cases, and we also derive the theoretical conditions for these results. Finally, we formulate an optimization problem to explore the cost difference between the optimal incentive protocols with and without errors, and further design an algorithm to obtain the numerical solution that minimizes this difference.
\end{abstract}

\begin{IEEEkeywords}
Evolutionary game theory, optimal control theory, cooperative behavior, Hamilton-Jacobi-Bellman equation, payoff-observation errors.
\end{IEEEkeywords}

\section{Introduction}
\IEEEPARstart{C}{ooperation} is fundamental to the functioning of multi-agent systems, enabling both humans and autonomous agents, e.g. robots, to collectively address challenging tasks in complex environments~\cite{Ferber1999, KrausAI1977, MintzBCHAOS25, Nguyencyb20}. Despite its importance, fostering and sustaining cooperation among self-interested agents remains a significant challenge, particularly in multi-agent systems~\cite{Tangcyb14, VincentCUP, AxelrodBBNY, RandTCS, RiehlARC18, LiTCNS18, Tangtnse24, Tancyb25}. Over the past few decades, evolutionary game theory has proven to be a powerful mathematical framework for analyzing the dynamics of cooperation in large populations of agents~\cite{Shicyb25, GlaubitzPNAS24, Sunhb2022, HilbePNAS13, Zhangtnse25, ShiJtnse24, Dingcyb21}. A quintessential model within this framework is the evolutionary Prisoner's Dilemma game (PDG), a paradigm that elucidates the subtle interplay between prosocial cooperation and selfish defection, providing foundational insights into  understanding how cooperation emerges within gaming environments ~\cite{Sunpnas2022, Smith1982, Gintis2009, Zhutac2022,  LiNC2020, Liucyb24}.

In this framework, incentive-based control policy is frequently employed to sustain collective cooperation in such multi-agent systems~\cite{HanINTER15, Sasaki2012, SunTNSE2023, Riehl12018, Vasconcelos2013, Paarporn2018, Fang2019PRSA, Wang2022JRCS, WangCNSNS2019}. Institutions typically use the \emph{combined incentives} to encourage cooperation while discouraging defection~\cite{Wang2025TAC, Lupre2024}. 
For example, in an emissions trading system, companies that exceed their carbon quotas must either pay fines or purchase additional permits, while those that reduce their emissions below the allocated limit can sell their surplus permits for profit or receive tax incentives~\cite{Schmalensee2017}.
However, given the costs associated with providing such incentives~\cite{Duong2021, Duong24BMB,ChenInterface2015}, there is increasing interest in designing policies that not only promote cooperation but also minimize the total implementation cost. It is worth noting that previous studies considered fixed incentive values that were independent of the actual system state (i.e., the proportion of cooperators) during the evolutionary process, leading to suboptimal incentive schemes~\cite{ChenInterface2015}. Recent studies have explored optimal incentive schemes with minimal costs by fully considering system state information, with promising results via optimal control approaches. For instance, Wang~\emph{et al.} designed the optimal protocols for institutional rewards, punishments, and combined incentives~\cite{WangCNSNS2019, Wang2022JRCS, Wang2025TAC}, and these protocols ensure that the total cost for the system, from the initial cooperation state to the desired cooperation state, is minimized in both well-mixed and structured populations.

Recent studies on evolutionary game theory have increasingly incorporated factors such as noisy payoffs \cite{WangG2024PNAS, FujimotoY2023PNAS}, imperfect information \cite{WangX2023NC,Schmid2023NC}, and perceptual errors \cite{Zhang2024eLifE, Barfuss2022PRE} to better approximate realistic decision-making environments and to examine how uncertainty influences the dynamics of cooperation. However, most prior research on the emergence of cooperation assumes that agents make perfectly accurate decisions during the process of interactions. In practice, this assumption is often unrealistic, as they frequently commit operational mistakes, such as incorrectly recording or misjudging the payoffs of their counterparts, thereby leading to \emph{payoff-observation errors}. These errors arise due to various causes, including emotional bias, time pressure, and complex environments, and they are pervasive in real-world systems \cite{ChenInterface2015, Alventosa2021}.  For instance, misjudging a competitor's pricing strategy in a competitive market can result in suboptimal decisions and financial losses \cite{Camerer1992}. Similarly, 
 in clinical decision-making, physicians often estimate treatment success based on limited data, and they may
 systematically overestimate or underestimate true medical emergencies \cite{Patel2002}. In our framework, these perceptual inaccuracies are captured by two behaviorally interpretable parameters, which describe the extent to which an agent underestimates or overestimates the other agents' payoff during strategy updating. Despite these observations, it remains unclear how payoff-observation errors influence the emergence of cooperation and how incentive-based control policies can be effectively implemented in the error-prone environments.
 
In this paper, we investigate the role of incentive-based control in promoting cooperation within the context of PDG in the presence of payoff-observation errors, a common form of individual errors. Here, individual errors represent perception and implementation inaccuracies made by agents \cite{ChenInterface2015}.  We derive the dynamical equation governing the evolution of cooperation in this error-prone gaming environment, and analyze the stability of equilibria for the dynamical equation to determine the incentive level required to guide the system towards a desired cooperative state. Subsequently, we use optimal control theory \cite{Evans2005, Geering2007} to construct an index function that quantifies the implementation cost of incentives, and  derive the optimal incentive protocols leading to the minimal cost for the evolution of cooperation.  Furthermore, we report a somehow counterintuitive and surprising finding: we identify the theoretical conditions, under which the presence of payoff-observation errors results in a lower cumulative cost induced by the optimal incentive protocol compared to the case without error, and this theoretical finding is also supported by numerical calculations. Finally, we formulate an optimization problem to investigate under what conditions the optimal cost in the presence of errors can exceed that in the error-free scenario to the greatest extent, and design an algorithm to obtain the numerical solution. The main contributions of our work are listed as follows:
\begin{itemize}
	\item  We formulate a systematic framework to study the evolutionary dynamics of population games in game-theoretic environments with payoff-observation errors.
	\item By using the Hamilton-Jacobi-Bellman equation (HJB equation), we analytically derive the optimal incentive protocols with minimal cost in the presence of payoff-observation errors under weak selection. 
	\item We  extend the analysis to the non-weak selection case, and show numerically that the cumulative cost of the optimal protocol decreases monotonically as the strength of selection increases, while the optimal protocol consistently outperforms alternative control schemes in terms of implementation cost.
	\item We find that payoff-observation errors can reduce the implementation cost compared with error-free case, and we further derive the corresponding theoretical conditions for this finding.
	\item We formulate an optimization problem to explore the difference between the optimal incentive costs with and without errors, and design an algorithm to obtain the numerical solution that minimizes this difference.
\end{itemize}
The rest of this paper is organized as follows. In Section \ref{sec2}, we formalize the problem formulation. In Section \ref{sec3}, we present the theoretical analysis in a gaming environment with payoff-observation errors, including the existence and stability analysis of equilibria, optimal incentive protocol for cooperation, 
comparison analysis of optimal incentive protocols between with and without payoff-observation errors, and the impact of errors on the cumulative cost induced by the optimal incentive protocols. Then, numerical results are presented in Section \ref{sec4} to verify the obtained theoretical analysis. Finally, our conclusions are drawn in Section \ref{sec5}. 

\textbf{Notation:} The following notations are used throughout the paper. We denote the set of non-negative and strictly positive integer numbers by $\mathbb{N}$ and $\mathbb{N_{+}}$, respectively. Let $\mathbb{R}=(-\infty, +\infty)$ be the real number set, and $\mathbb{R}^m$ represents the space of real-valued $m$-dimensional vectors. For any vector $x\in\mathbb{R}^m$, $\Vert x\Vert_2$ denotes its $\text{L2}$ norm (Euclidean norm).

\begin{figure*}[!t]
	\begin{center}
		\includegraphics[width=6in]{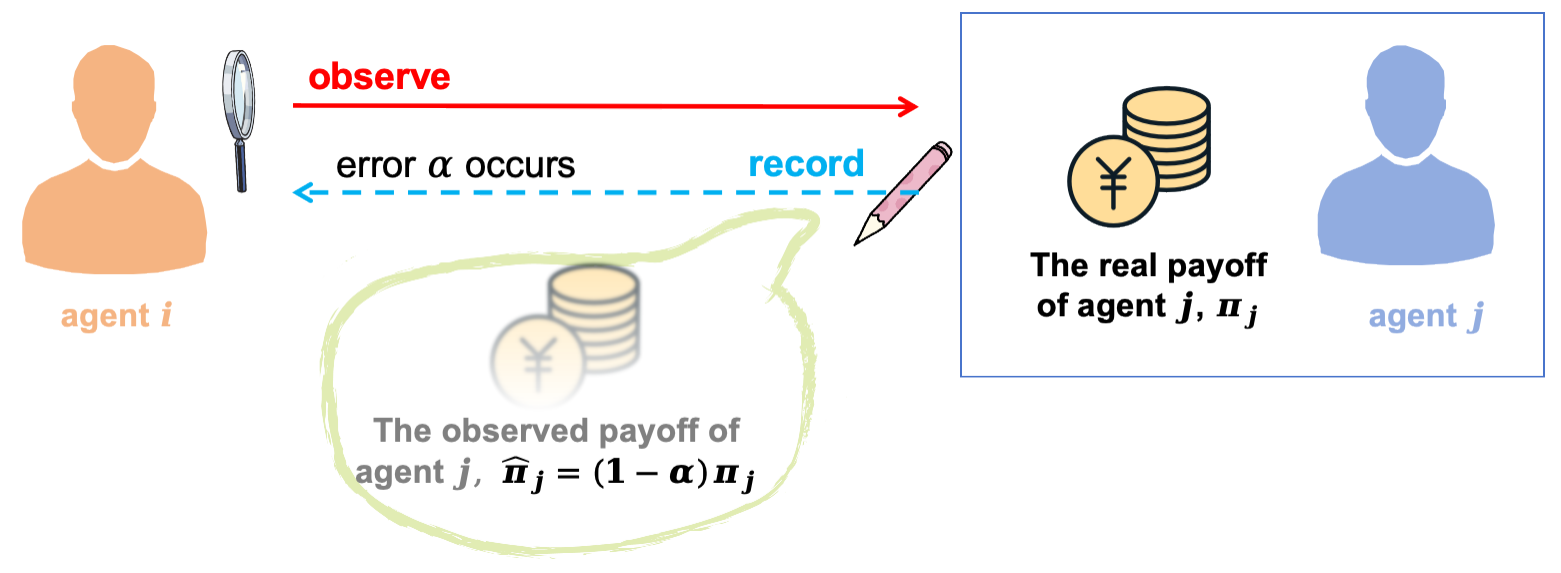}
		\caption{Evolutionary Prisoner's Dilemma Game with the combined incentive in a gaming environment with payoff-observation errors. This figure illustrates the case where agent $i$ attempts to observe the payoff of agent $j$, but the operational factors (e.g., imprecisely recording the payoff value) may cause a discrepancy between the observed payoff $\Pi_{j}$ and real payoff $\hat{\Pi}_{j}$ of agent $j$, where $\hat{\Pi}_{j}=(1-\alpha)\Pi_{j}$ and $\alpha \in[-1, 1]$ represents the observed error parameter.}\label{fig1}
	\end{center}
\end{figure*}

\section{Problem formulation}\label{sec2}
\subsection{Prisoner's dilemma game with the combined incentive}
We study evolutionary games in a well-mixed population of $n\in\mathbb{N_{+}}$ agents, and all of these agents are denoted by the set $\mathcal{V}=\{1, \dots, n\}$. Each agent plays the evolutionary PDG with its neighbors, who chooses either to be a cooperator by adopting cooperation ($C$) or to be a defector by selecting defection ($D$), and it receives payoffs based on the following payoff matrix:
\begin{equation}
	A(u)=\begin{bmatrix}
		b-c&-c\\
		b&0
	\end{bmatrix}
	+
	\begin{bmatrix}
		a_{11}(u) &a_{12}(u)\\
		a_{21}(u)&a_{22}(u)
	\end{bmatrix}.
\end{equation}\label{eq1}
This matrix consists of two submatrices. The first submatrix represents the basic payoff from each interation without additional regulatory mechanisms, while the second one shows the combined incentive value provided by an institution. Specifically, in the first submatrix, corresponding to a so-called `Donation Game', 
$b-c$, $-c$, $b$, and $0$  are  respectively the payoffs for $C$-against-$C$, $C$-against-$D$, $D$-against-$C$, and $D$-against-$D$, where $b>c>0$. It can be observed  that  defection for each agent is the dominant choice in this game. However, if both agents choose $C$, then they could  yield a higher  payoff. But when both choose to defect, both get nothing. This inevitably leads to a social dilemma of cooperation. 

The second term is called the incentive matrix. To tackle the dilemma mentioned above,  we consider a top-down incentive-providing institution to implement combined incentives for promoting cooperation. In this matrix, $a_{11}(u)=pu_{R}$, $a_{12}(u)=pu_{R}$, $a_{21}(u)=(1-p)u_{P}$, and $a_{22}(u)=(1-p)u_{P}$ are the values of incentives received from the institution for $C$-against-$C$, $C$-against-$D$, $D$-against-$C$, and $D$-against-$D$ in the game respectively, where $p\in[0, 1]$ is a rewarding preference for the incentive-providing institution, and each $p$-value corresponds to a form of the combined incentive. $u_{R}$ and $u_{P}$ respectively represent the per capita rewarding and punishing incentives.  For convenience, we assume $u_{P}=u_{R}=u$ in this study. Specifically,  when playing with a $C$-agent,  the payoffs of each cooperator and defector are respectively
$\pi_{C}^C=b-c+a_{11}(u)=b-c+pu_{R}$
and
$\pi_{D}^C=b+a_{21}(u)=b-(1-p)u_{P}$, where $C$-agent is the one plays strategy $C$.
By contrast, when playing against a $D$-agent, the payoffs of each cooperator and  defector are respectively
$\pi_C^D=-c+a_{12}(u)=-c+pu_{R}$
and
$\pi_D^D=0+a_{22}(u)=-(1-p)u_{P}$, where $D$-agent is the one plays strategy $D$.

\begin{Remark}\label{Remark1}
	$p=1$ and $p=0$ separately represent pure reward and punishment.
\end{Remark}

\subsection{Evolutionary dynamics of cooperation in a gaming environment with payoff-observation errors}

In a well-mixed population of size $n$, if there are $k\leq n$ cooperators and $n-k$ defectors, then the average payoffs  of cooperators and defectors can be respectively  written as  \cite{Schuster1983, Hofbauer1998}
\begin{equation}\label{eq2}
	\begin{aligned}
		\Pi_C&=\frac{\pi_{C}^C(k-1)+\pi_C^D(n-k)}{n-1}\\
		&=\frac{(b-c+pu_{R})(k-1)-(c-pu_{R})(n-k)}{n-1},
	\end{aligned}
\end{equation}
and 
\begin{equation}\label{eq3}
	\begin{aligned}
		\Pi_D&=\frac{\pi_{D}^Ck+\pi_D^D(n-1-k)}{n-1}\\
		&=\frac{\big[b-(1-p)u_{P}\big]k-(1-p)u_{P}(n-1-k)}{n-1}.
	\end{aligned}
\end{equation}

As the population game proceeds in time, each  agent has the same chance to update its strategy by comparing the above-mentioned average payoff with the one obtained by a neighboring agent. Accordingly, at each time step two neighboring $i$ and $j$ agents with $i, j \in\mathcal{V}$ are randomly selected, and calculate their  payoffs difference  $\Pi_{j}-\Pi_{i}$. Then, agent $i$ adopts the strategy of the neighboring $j$ with
the probability  defined by the Fermi function 
$
W_{i\rightarrow j}=\frac{1}{1+e^{-\omega(\Pi_{j}-\Pi_{i})}}
=W_{i\rightarrow j}(0)+\frac{\textrm{d}W_{i\rightarrow j}(0)}{\textrm{d}\omega}\omega+O(\omega^2)\approx\frac{1}{2}+\omega\frac{\Pi_{j}-\Pi_{i}}{4},
$
where $\omega\in[0,1]$ measures the strength of selection~\cite{Szabo1998}.  In the $\omega\rightarrow 0$ limit, called the weak selection limit,  agent $i$ will stay with its own strategy or  imitate the strategy of neighbor $j$ with equal probabilities.
When $\omega>0$, the more successful neighbor $j$ is, the more likely it is that agent $i$ will adopt the strategy of neighbor $j$.  It is worth pointing out that in this study, we focus on the case of weak selection to have some theoretical predictions. We stress that the weak selection assumption in our model is reasonable, because the payoff contribution is merely a part of the elements impacting the learning process and it makes sense to quantify the effect size relative to the case that $\omega$ is sufficiently small but nonzero.

Although each agent updates its strategy mainly based on the payoff difference between itself and a randomly selected neighbor, it  may have an incorrect observation of the neighbor's payoff (see Fig.~\ref{fig1}),  which affects its strategy update rule. 
Assume that during the strategy update process, $C$-agent observes a deviation in the payoff of $D$-agent, with the observed error parameter given by $\alpha=\frac{\Pi_{D}- \hat{\Pi}_{D}}{\Pi_{D}}\in[-1,1]$, where $\hat{\Pi}_{D}=(1-\alpha)\Pi_{D}$ is the payoff observed by $C$-agent for $D$-agent.  Similarly, $D$-agent observes a deviation in the payoff of $C$-agent, with the observed error parameter given by $\beta=\frac{\Pi_{C}-\hat{\Pi}_{C}}{\Pi_{C}}\in[-1,1]$, where $\hat{\Pi}_{C}=(1-\beta)\Pi_{C}$ is the payoff observed by $D$-agent for $C$-agent. Therefore,  $C$-agent adopts the strategy of  $D$-neighbor with the probability
\begin{equation}\label{eq4}
	\begin{aligned}
		\widetilde{W}_{C\rightarrow D}=\frac{1}{1+e^{-\omega[\hat{\Pi}_{D}-\Pi_{C}]}}\approx\frac{1}{2}+\omega\frac{\hat{\Pi}_{D}-\Pi_{C}}{4},	
	\end{aligned}
\end{equation}
and 
$D$-agent adopts the strategy of $C$-neighbor with the probability
\begin{equation}\label{eq5}
	\begin{aligned}
		\widetilde{W}_{D\rightarrow C}=\frac{1}{1+e^{-\omega[\hat{\Pi}_{C}-\Pi_{D}]}}\approx\frac{1}{2}+\omega\frac{\hat{\Pi}_{C}-\Pi_{D}}{4}.
	\end{aligned}
\end{equation}

For large $n\rightarrow \infty$, we derive the mean-field equation of the dynamical changes of the fraction of cooperators, which is given by  (for details see
Appendix A)~\cite{Traulsen2005,Traulsen2009}:
\begin{equation}\label{eq6}
	\begin{aligned}
		\dot{x}(t)&=x(1-x)\widetilde{W}_{D\rightarrow C}-x(1-x)\widetilde{W}_{C\rightarrow D}\\
		&=x(1-x)\Big[\frac{1}{1+e^{-\omega(\hat{\Pi}_{C}-\Pi_{D})}}-\frac{1}{1+e^{-\omega(\hat{\Pi}_{D}-\Pi_{C})}}\Big]\\
		&=x(1-x)\frac{e^{-\omega(\hat{\Pi}_{D}-\Pi_{C})}[1-e^{-\omega(-\hat{\Pi}_{D}+\Pi_{C}+\hat{\Pi}_{C}-\Pi_{D})}]}{[1+e^{-\omega(\hat{\Pi}_{D}-\Pi_{C})}][1+e^{-\omega(\hat{\Pi}_{C}-\Pi_{D})}]} \\
		&=x(1-x)\psi(\omega)e^{-\omega \Delta},
	\end{aligned}
\end{equation}
where $\psi(\omega)=\frac{e^{-\omega(\hat{\Pi}_{D}-\Pi_{C})}}{[1+e^{-\omega(\hat{\Pi}_{D}-\Pi_{C})}][1+e^{-\omega(\hat{\Pi}_{C}-\Pi_{D})}]}$,  $\Delta=-\hat{\Pi}_{D}+\Pi_{C}+\hat{\Pi}_{C}-\Pi_{D}=(2-\beta)\Pi_{C}-(2-\alpha)\Pi_{D}$, and  $x \in[0,1]$ denotes the  proportion of cooperators in the whole population. Combining \eqref{eq2} and \eqref{eq3}, 
the above-mentioned system equation in the limit of $\omega\rightarrow 0$ can be rewritten as
\begin{equation}\label{eq7}
	\dot{x}(t)=F(x, u, t)=\Gamma x(1-x),
\end{equation}
where $\Gamma=\frac{\omega}{4}\Delta=\frac{\omega}{4}[(2-\beta)\Pi_{C}-(2-\alpha)\Pi_{D}]=\frac{\omega}{4}[(\alpha-\beta)bx-(2-\beta)c+[2-\alpha+(\alpha-\beta)p]u]$.

\begin{Remark}\label{Remark2}
	It is worth noting that the error parameters $\alpha$ and $\beta$ are  restricted to the interval $[-1,1]$. This is because although individuals may make observational errors due to operational mistakes, these errors are generally small and should remain within a reasonable range. We thus assume the observed payoff will not exceed twice the real payoff of targeted agent, as this would allow agents to recognize that the observation is problematic.
\end{Remark}

\begin{Remark}\label{Remark3}
From Eqs.~\eqref{eq6} and ~\eqref{eq7}, one can see that both for finite $\omega$ and under weak selection $\omega\rightarrow 0$, the system admits only the two boundary equilibria $x^* = 0$ and $x^* = 1$. The existence and stability of interior equilibria are determined solely by $\Delta$. The selection strength $\omega$ affects only the speed of the system evolution, but does not alter the qualitative phase structure or the location of equilibria. An interior equilibrium arises only in the degenerate case $\Delta = 0$; otherwise, the direction of the flow and the stability of the boundary equilibria are entirely governed by the sign of $\Delta$.
\end{Remark}

\begin{assumption}\label{assumption1}
The population size $n$, the payoff matrix $A(u)$, the strength of selection $\omega$, and the error parameters $\alpha$ and $\beta$ are time-invariant. In contrast, the proportion of cooperators $x$ in the population changes with time $t$. 
\end{assumption}

\subsection{Optimal control problem}

Since providing incentives is costly for institutions, the goal of optimization problems for the combined incentive is to minimize the implementation cost of incentives required for supporting cooperation, and find the explicit expressions of optimal incentive protocol $u^\ast$. To this end, we formulate an optimal control problem for the combined incentive as follows: 
\begin{eqnarray}\label{eq8}
	&\min\,\,&J=\frac{1}{2}\int^{t_{f}}_{t_0}\|G(x, u, t)\|_2^2\textrm{d}t, \nonumber\\
	& s.t.&\quad  \left\{\begin{array}{lc}
		\dot{x}(t)=F(x, u, t),  \\
		x(t_0)=x_{0}, \\
		x(t_{f})=1-\delta,
	\end{array}\right.
\end{eqnarray}
where 
\begin{eqnarray}\label{eq9}
	\begin{aligned}
		G(x, u, t)&=n(n-1) xpu_R+n(n-1)(1-x)(1-p)u_P\\
		&=n(n-1)u [px+(1-p)(1-x)].
	\end{aligned}
\end{eqnarray}
Here, $(n-1)xpu_R$ represents the amount of rewarding incentive that each agent receives from the institution  after interacting with all of its $C$-neighbors, and $(n-1)(1-x)(1-p)u_P$ denotes the amount of punishing incentive imposed on each agent after interacting with all of its $D$-neighbors, where $(n-1)x$ and $(n-1)(1-x)$ are the average numbers of $C$-neighbors and $D$-neighbors for each agent with $n-1$ neighbors, respectively. Consequently,  the incentive cost incurred by the institution for an agent at time $t$ is given by $(n-1)[xpu_R+(1-x)(1-p)u_P]$. When all $n$ agents complete their interactions, the total cost of incentives at time $t$ is $n(n-1)[xpu_R+(1-x)(1-p)u_P]$.

\begin{figure}[!t]
	\begin{center}
		\includegraphics[width=3.5in]{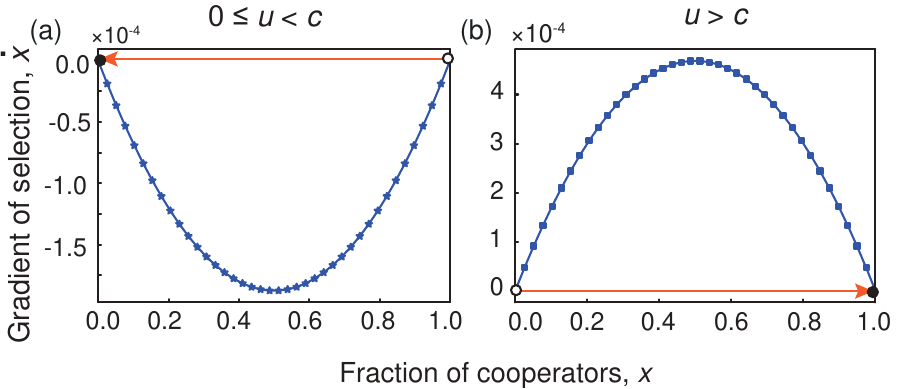}
		\caption{The gradient of selection in dependence on the proportion of cooperators in the whole population under the condition  $\alpha=\beta$. Here, panel ($a$) satisfies  $0\leq u<c$, while the other case is shown in panel ($b$). Solid (open) circles represent stable (unstable) equilibria. 
			The values of $u$ are $u=0.8$ ($a$) and  $u=1.5$ ($b$). Parameters: $n=100$, $b=2$, $c=1$, $\alpha=0.5$, $p=0.5$, $x_{0}=0.5$, $\beta=0.5$, and $\omega=0.01$.}\label{fig2}
	\end{center}
\end{figure}

For the optimal control problem  mentioned above, the per capita incentive $u$ is regarded as the control variable, and the cumulative cost $J$ as the objective function. The cost function $J$ describes the cumulative cost on average from the initial time $t_0$  to the terminal time $t_{f}$ for the dynamical system \eqref{eq7}.  For convenience,  the initial time is set as $0$ (i.e., $t_0=0$).  Herein, we denote the initial cooperation state (i.e., the initial proportion of cooperators) in the population by $x(t_0)=x_{0}$, and denote the desired or terminal cooperation level (i.e., the desired proportion of cooperators) in the population by $x(t_{f})$.  With the aim of exploring the optimally combined incentive protocol $u^*$ and the corresponding optimal function $J^\ast=\frac{1}{2}\int^{t_{f}}_{t_0}\|G(x, u^{\ast}, t)\|_2^2\textrm{d}t$ over the time interval $[t_0, t_{f}]$, we now impose two assumptions for the optimal control problem.
\begin{assumption}\label{assumption3}
	The terminal time  $t_{f}$ is  not fixed.
\end{assumption}

\begin{assumption} \label{assumption4}
	The terminal state $x(t_{f})$ is set to be $1-\delta$ (i.e., $x(t_{f})=1-\delta>x_{0}$), where $\delta$ is the parameter that determines the expected cooperation level at $t_{f}$.
\end{assumption}

Assumption \ref{assumption3} stipulates that the optimal function $J^\ast$ is independent of $t$, that is,  $\frac{\partial J^\ast}{\partial t}=0$. Assumption \ref{assumption4} means that  the desired cooperation level $x(t_{f})$ can be reached at the terminal time $t_{f}$. Therefore, the quantity $J$ can be regarded as the objective functional of calculating the optimally combined incentive protocol $u^\ast$ with the lowest  cumulative cost.

\section{THEORETICAL ANALYSIS}\label{sec3}

\subsection{Existence and stability analysis of equilibria}
In this subsection, we analyze the existence and asymptotic stability of equilibria for the system equation obtained in  \eqref{eq7} under three different scenarios regarding the 
relationship between $\alpha$ and $\beta$, which are respectively $\alpha=\beta$, $\alpha>\beta$, and $\alpha<\beta$. The related results of $\alpha>\beta$ are presented in Lemma \ref{lem1}, those of $\alpha=\beta$ in Lemma \ref{lem2}, and  those of $\alpha<\beta$ in Lemma \ref{lem3}.

%\noindent \begin{theorem}\label{thm1}
	\begin{lemma}\label{lem1}
		For $\alpha=\beta$ and any $x_0\in (0, 1)$, the following statements hold:
		\begin{enumerate}
			\item If $0\leq u< c$, there  exist two equilibria, that is, $x^\ast=1$ and $x^\ast=0$, where  $x^\ast=1$ is unstable and $x^\ast=0$ is asymptotically stable.
			\item If $u> c$, there  exist two equilibria, that is, $x^\ast=1$ and $x^\ast=0$, where $x^\ast=1$ is asymptotically stable and $x^\ast=0$ is unstable.
		\end{enumerate}
	\end{lemma}
	\noindent \begin{proof}
		Solving $F(x, u, t)=0$, it is straightforward to check that Eq.~\eqref{eq7} admits two equilibria: $x^\ast=1$ and  $x^\ast=0$. The partial derivative of  $F(x, u, t)$  with respect to $x$ is given by
		\begin{equation}\label{eq10}
			\begin{aligned}
				\frac{\partial F}{\partial x}=\frac{\omega(2-\beta)(u-c)}{4}(1-2x),
			\end{aligned}
		\end{equation}
		which at these two equilibria, yields that  $\frac{\partial F}{\partial x}\big|_{x^\ast=1}=-\frac{\omega(2-\beta)(u-c)}{4}$ and $\frac{\partial F}{\partial x}\big|_{x^\ast=0}=\frac{\omega(2-\beta)(u-c)}{4}$.
		For $0\leq u<c$, we obtain that $\frac{\partial F}{\partial x}\big|_{x^\ast=1}>0$ and $\frac{\partial F}{\partial x}\big|_{x^\ast=0}<0$ all $t\geq0$, indicating that $x^\ast=1$ is unstable and  $x^\ast=0$ is asymptotically stable. Conversely, for $u>c$, we have $\frac{\partial F}{\partial x}\big|_{x^\ast=1}<0$ and $\frac{\partial F}{\partial x}\big|_{x^\ast=0}>0$ for all $t\geq0$, implying that  $x^\ast=1$ is asymptotically stable and $x^\ast=0$ is unstable.
	\end{proof}

	\begin{figure}[!t]
		\begin{center}
			\includegraphics[width=3.6in]{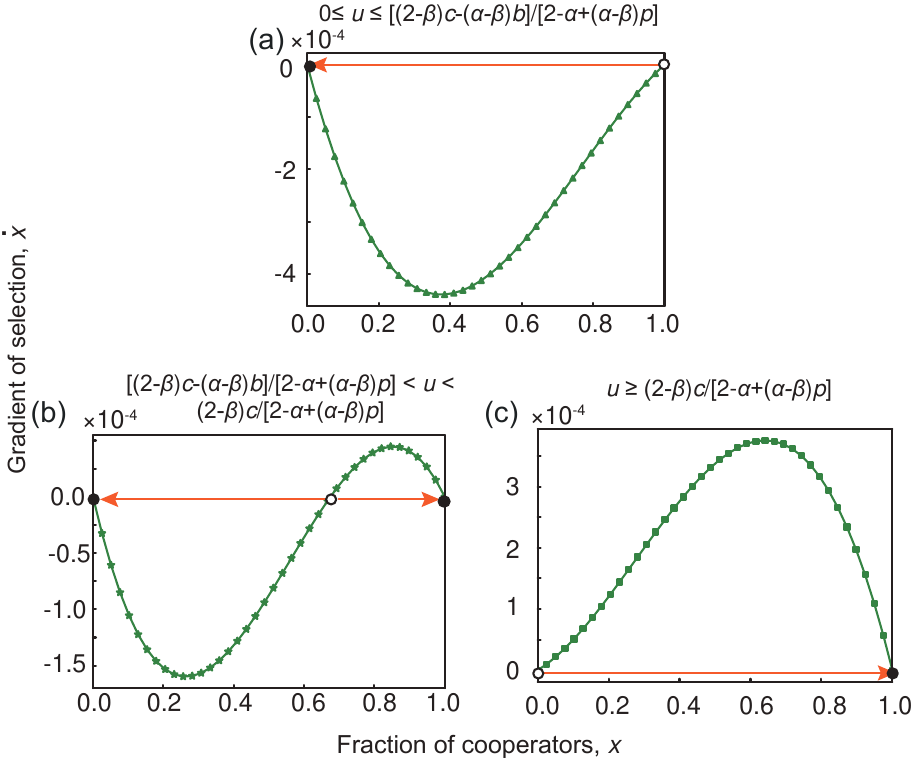}
			\caption{The gradient of selection in dependence on the proportion of cooperators in the whole population under the condition $\alpha>\beta$. Here, panel ($a$) satisfies  $0\leq  u\leq\frac{(2-\beta)c-(\alpha-\beta)b}{2-\alpha+(\alpha-\beta)p}$, panel ($b$) satisfies  $\frac{(2-\beta)c}{2-\alpha+(\alpha-\beta)p}<u<
				\frac{(2-\beta)c-(\alpha-\beta)b}{2-\alpha+(\alpha-\beta)p}$, and the another case is shown in panel ($c$).  Solid (open) circles represent stable (unstable) equilibria. 
				The values of $u$ are  $u=0.5$ ($a$), $u=0.8$ ($b$), and  $u=1.2$ ($c$). 
				Parameters: $n=100$, $b=2$, $c=1$, $\alpha=0.5$, $p=0.5$, $x_{0}=0.5$, $\beta=0.1$, and $\omega=0.01$.}\label{fig3}
		\end{center}
	\end{figure}

	\begin{lemma}\label{lem2}
		For $\alpha>\beta$ and any  $x_0\in (0, 1)$, the following statements hold:
		\begin{enumerate}
			\item If $0\leq u\leq\frac{(2-\beta)c-(\alpha-\beta)b}{2-\alpha+(\alpha-\beta)p}$,  there exist two equilibria, that is, $x^\ast=0$ and $x^\ast=1$, where $x^\ast=0$ is asymptotically stable and $x^\ast=1$ is unstable.
			\item If $\frac{(2-\beta)c-(\alpha-\beta)b}{2-\alpha+(\alpha-\beta)p}<u<\frac{(2-\beta)c}{2-\alpha+(\alpha-\beta)p}$, there exist three equilibria, that is, $x^\ast=0$, $x^\ast=1$, and $x^\ast=\bar{x}$, where $x^\ast=0$ and $x^\ast=1$ are asymptotically stable, while $x^\ast=\bar{x}$ is unstable, with $\bar{x}=\frac{(2-\beta)c-[2-\alpha+(\alpha-\beta)p]u}{(\alpha-\beta)b}$.
			\item If $u\geq\frac{(2-\beta)c}{2-\alpha+(\alpha-\beta)p}$, there exist two equilibria, that is, $x^\ast=0$ and $x^\ast=1$, where $x^\ast=1$ is asymptotically stable and $x^\ast=0$ is unstable.
		\end{enumerate}
	\end{lemma}
	\noindent \begin{proof}
		Solving $F(x, u, t)=0$, we find that Eq.~\eqref{eq7} has two equilibria: $x^\ast=0$, and $x^\ast=1$, along with a possible interior equilibrium point $x^\ast=\bar{x}$. In addition, the partial derivative of $F(x, u, t)$  with respect to $x$  is given by 
		\begin{equation}\label{eq11}
			\begin{aligned}
				&\frac{\partial F}{\partial x}=\Gamma(1-2x)+ \frac{\omega(\alpha-\beta)b}{4}x(1-x)\\
				&=\omega\frac{(\alpha-\beta)bx-(2-\beta)c+[2-\alpha+(\alpha-\beta)p]u}{4}\\
				&\times(1-2x)+\frac{\omega(\alpha-\beta)b}{4}x(1-x).
			\end{aligned}
		\end{equation}
		In particular, 	at these equilibrium points, we have
		\begin{subequations}\label{eq12}
			\begin{align}
				&\frac{\partial F}{\partial x}\big|_{x^\ast=1}=-\omega\frac{(\alpha-\beta)b-(2-\beta)c+[2-\alpha+(\alpha-\beta)p]u}{4},\label{eq:12a}\\
				&\frac{\partial F}{\partial x}\big|_{x^\ast=0}=\omega\frac{[2-\alpha+(\alpha-\beta)p]u-(2-\beta)c}{4}, \label{eq:12b}\\
				&\frac{\partial F}{\partial x}\big|_{x^\ast=\bar{x}}=\frac{\omega(\alpha-\beta)b}{4}\bar{x}(1-\bar{x}). \label{eq:12c}
			\end{align}
		\end{subequations}

		For $0\leq u\leq\frac{(2-\beta)c-(\alpha-\beta)b}{2-\alpha+(\alpha-\beta)p}$, one checks $x^\ast=\bar{x}\geq1$,  implying that the interior equilibrium point does not exist. Thus, 
		Eq.~\eqref{eq7}  again has only two equilibria: $x^\ast=0$, and $x^\ast=1$. At these points, we further have that $\frac{\partial F}{\partial x}\big|_{x^\ast=1}>0$  and $\frac{\partial F}{\partial x}\big|_{x^\ast=0}<0$ for all $t\geq0$. Therefore,  $x^\ast=1$ is unstable, and  $x^\ast=0$ is asymptotically stable.
		Next, for $\frac{(2-\beta)c-(\alpha-\beta)b}{2-\alpha+(\alpha-\beta)p}<u<\frac{(2-\beta)c}{2-\alpha+(\alpha-\beta)p}$, 
		one gets $x^\ast=\bar{x}\in(0, 1)$, and accordingly, Eq.~\eqref{eq7} has three equilibria: $x^\ast=0$, $x^\ast=\bar{x}$, and $x^\ast=1$, which at these three equilibria, yields that $\frac{\partial F}{\partial x}\big|_{x^\ast=1}<0$, $\frac{\partial F}{\partial x}\big|_{x^\ast=\bar{x}}>0$,  and $\frac{\partial F}{\partial x}\big|_{x^\ast=0}<0$ for all $t\geq0$. Therefore,  $x^\ast=1$ is asymptotically stable,  $x^\ast=0$ is asymptotically stable, and $x^\ast=\bar{x}$ is unstable. 
		Last, for $u\geq\frac{(2-\beta)c}{2-\alpha+(\alpha-\beta)p}$,
		we have $x^\ast=\bar{x}\leq0$, indicating that the interior equilibrium point does not exist. As a result, Eq.~\eqref{eq7} has only two equilibria: $x^\ast=0$ and $x^\ast=1$. At these equilibria, one yields that $\frac{\partial F}{\partial x}\big|_{x^\ast=1}<0$ and $\frac{\partial F}{\partial x}\big|_{x^\ast=0}>0$  for all $t\geq0$, and thus,  $x^\ast=1$ is asymptotically stable and $x^\ast=0$ is unstable. 
	\end{proof}

	\begin{lemma}\label{lem3}
		For $\alpha<\beta$ and any  $x_0\in (0, 1)$, the following statements hold:
		\begin{enumerate}
			\item If $0\leq u\leq\frac{(2-\beta)c}{2-\alpha+(\alpha-\beta)p}$,  there exist two equilibria, that is, $x^\ast=0$ and $x^\ast=1$, where $x^\ast=0$ is asymptotically  stable and $x^\ast=1$ is unstable.
			\item If $\frac{(2-\beta)c}{2-\alpha+(\alpha-\beta)p}<u<
			\frac{(2-\beta)c-(\alpha-\beta)b}{2-\alpha+(\alpha-\beta)p}$, there exist three equilibria, that is, $x^\ast=0$, $x^\ast=1$, and $x^\ast=\hat{x}$, where $x^\ast=0$ and $x^\ast=1$ are unstable, while $x^\ast=\bar{x}$ is asymptotically stable.
			\item If $u\geq\frac{(2-\beta)c-(\alpha-\beta)b}{2-\alpha+(\alpha-\beta)p}$, there exist two equilibria, that is, $x^\ast=0$ and $x^\ast=1$, where  $x^\ast=1$ is asymptotically stable and $x^\ast=0$ is unstable.
		\end{enumerate}
	\end{lemma}
	\noindent \begin{proof}
		Similar to Lemma \ref{lem2}, Eq.~\eqref{eq7} has two equilibria: $x^\ast=0$ and $x^\ast=1$, as well as one potentially existing interior equilibrium point $x^\ast=\bar{x}$. Furthermore, the partial derivative of $F(x, u, t)$  with respect to $x$ is defined in \eqref{eq11}.
		
		For  $0\leq u\leq\frac{(2-\beta)c}{2-\alpha+(\alpha-\beta)p}$, one checks $x^\ast=\bar{x}\leq0$, implying that \eqref{eq7} has only two equilibria: $x^\ast=0$ and $x^\ast=1$. At these two equilibria, we have 
		that $\frac{\partial F}{\partial x}\big|_{x^\ast=1}>0$  and $\frac{\partial F}{\partial x}\big|_{x^\ast=0}<0$ for $\forall \,\, t\geq0$. Therefore,  $x^\ast=1$ is unstable, and  $x^\ast=0$ is asymptotically stable.
		Next, for  $\frac{(2-\beta)c}{2-\alpha+(\alpha-\beta)p}<u<
		\frac{(2-\beta)c-(\alpha-\beta)b}{2-\alpha+(\alpha-\beta)p}$, one gets $x^\ast=\bar{x}\in(0, 1)$, meaning that \eqref{eq7} has three equilibria:  $x^\ast=0$, $x^\ast=\bar{x}$, and $x^\ast=1$. At these three equilibria, we find that 
		$\frac{\partial F}{\partial x}\big|_{x^\ast=1}>0$, $\frac{\partial F}{\partial x}\big|_{x^\ast=\bar{x}}<0$, and $\frac{\partial F}{\partial x}\big|_{x^\ast=0}>0$ for $\forall \,\, t\geq0$. Therefore,  both $x^\ast=0$ and $x^\ast=1$ are unstable, while $x^\ast=\bar{x}$ is asymptotically stable. 
		Finally, for $u\geq\frac{(2-\beta)c-(\alpha-\beta)b}{2-\alpha+(\alpha-\beta)p}$, we have $x^\ast=\bar{x}\geq1$, implying that the interior equilibrium point does not exist. Consequently, \eqref{eq7} has only two equilibria: $x^\ast=0$, and $x^\ast=1$. At these two equilibria, one yields that $\frac{\partial F}{\partial x}\big|_{x^\ast=1}<0$ and $\frac{\partial F}{\partial x}\big|_{x^\ast=0}>0$  for $\forall \,\, t\geq0$, and thus,  $x^\ast=1$ is asymptotically stable and $x^\ast=0$ is unstable. 
	\end{proof}	
	
	\begin{figure}[!t]
		\begin{center}
			\includegraphics[width=3.6in]{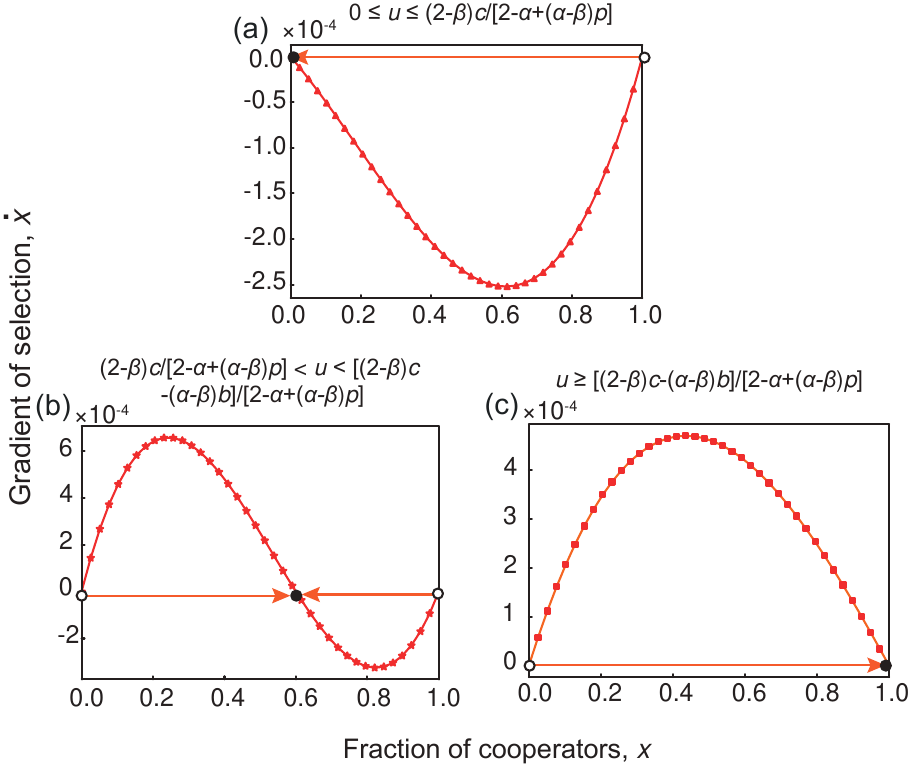}
			\caption{The gradient of selection in dependence on the proportion of cooperators in the whole population under the condition  $\alpha<\beta$, where panel ($a$) satisfies $0\leq u\leq\frac{(2-\beta)c}{2-\alpha+(\alpha-\beta)p}$, panel ($b$) satisfies $\frac{(2-\beta)c}{2-\alpha+(\alpha-\beta)p}<u<
				\frac{(2-\beta)c-(\alpha-\beta)b}{2-\alpha+(\alpha-\beta)p}$, and the another case is presented in panel ($c$). Solid (open) circles represent stable (unstable) equilibria. 
				The values of $u$ are  $u=0.8$ ($a$), $u=1.1$ ($b$), and  $u=1.6$ ($c$). 
				Parameters: $n=100$, $b=2$, $c=1$, $\alpha=0.5$, $p=0.5$, $x_{0}=0.5$, $\beta=0.7$, and $\omega=0.01$.}\label{fig4}
		\end{center}
	\end{figure}
	\begin{Remark}\label{Remark4}
		The conditions of  $0\leq u\leq\frac{(2-\beta)c-(\alpha-\beta)b}{2-\alpha+(\alpha-\beta)p}$ in lemma \ref{lem2} and $\frac{(2-\beta)c}{2-\alpha+(\alpha-\beta)p}<u<
		\frac{(2-\beta)c-(\alpha-\beta)b}{2-\alpha+(\alpha-\beta)p}$ in  lemma \ref{lem3}  imply that the inequality $(2-\beta)c>(\alpha-\beta)b$ holds, and we thus do not consider the case where $(2-\beta)c\leq(\alpha-\beta)b$ in this work. 
	\end{Remark}
	
From the above-mentioned Lemmas, when $\alpha=\beta$  (Lemma \ref{lem1}), the governing system equation with the combined incentive admits two equilibrium points, namely, $x^\ast=0$ and $x^\ast=1$. In this case, cooperation becomes more abundant than defection if and only if $u>c$. For the case of $\alpha>\beta$ (Lemma \ref{lem2}), the system equation exhibits three equilibrium points: 
	$x^\ast=0$, $x^\ast=1$, and $x^\ast=\bar{x}$. Here, 
	cooperation dominates defection if and only if $u\geq\frac{(2-\beta)c}{2-\alpha+(\alpha-\beta)p}$.  Similarly,  when  $\alpha<\beta$ (Lemma \ref{lem3}), the system equation again has three equilibrium points mentioned above, and  
	cooperation prevails if and only if  $u\geq\frac{(2-\beta)c-(\alpha-\beta)b}{2-\alpha+(\alpha-\beta)p}$. However, implementing the combined incentive incurs a cost, making it desirable to optimize an incentive-based control policy that ensures the establishment of cooperation at a lower cost.  This aspect will be explored in the next subsection.

	\subsection{Optimal incentive protocol for cooperation}
	Using the results presented in the previous subsection, we study the consequence of time-dependent combined incentive. We show  the optimally combined incentive protocol by solving the optimal control problem defined in \eqref{eq11}. Then, we analyze the amount of cumulative cost induced by the optimally combined incentive for the dynamical system to reach the expected terminal state from the initial state. The related results are presented in Theorem \ref{thm1}.
	
	\begin{theorem}\label{thm1}
		The optimally combined incentive protocol for the optimal control problem \eqref{eq11} is
		\begin{equation}\label{eq13}
			\begin{aligned}
				u^{\ast}=\frac{2[(2-\beta)c-(\alpha-\beta)bx]}{2-\alpha+(\alpha-\beta)p}.
			\end{aligned}
		\end{equation}
		With the optimally combined incentive protocol, the solution of the system~\eqref{eq10} is
		\begin{equation}\label{eq14}
			\begin{aligned}
				\dot{x}(t)=\omega\frac{(2-\beta)c-(\alpha-\beta)bx}{4}x(1-x),
			\end{aligned}
		\end{equation}
		and its cumulative cost is
		\begin{equation}\label{eq15}
			\begin{aligned}
				J^\ast=\phi[(2-\beta)c\mathcal{A}-(\alpha-\beta)b\mathcal{B}],
			\end{aligned}
		\end{equation}
		where
		$\phi=\frac{8n^2(n-1)^2}{\omega [2-\alpha+(\alpha-\beta)p]^2}>0$, $\mathcal{A}=(2p-1)^2(x_{0}-1+\delta)+p^2\ln\big(\frac{1-x_0}{\delta}\big)+(1-p)^2\ln\big(\frac{1-\delta}{x_0}\big)>0$, and
		$\mathcal{B}=(1-2p)(1-x_{0}-\delta)-\frac{(2p-1)^2}{2}(1-x_{0}-\delta)(1+x_{0}-\delta)+p^2\ln\big(\frac{1-x_0}{\delta}\big)$. 
	\end{theorem}
	\begin{proof} 
		The Hamiltonian function $H(x, u, t)$ of \eqref{eq8} is defined as
		\begin{equation}\label{eq16}
			\begin{aligned}
				&H(x, u, t)=\frac{1}{2}\|G(x, u, t)\|_2^2+\frac{\partial J^{\ast}}{\partial x}F(x, u,  t)\\
				&=\frac{1}{2}[G(x, u, t)]^2+\frac{\partial J^{\ast}}{\partial x}F(x, u,  t)\\
				&=\frac{1}{2}n^2(n-1)^2u^2 [px+(1-p)(1-x)]^2\\
				&+\frac{\partial J^{\ast}}{\partial x}\frac{\omega x(1-x)}{4}\big\{(\alpha-\beta)bx-(2-\beta)c\\
				&+[2-\alpha+(\alpha-\beta)p]u\big\},
			\end{aligned}
		\end{equation}
		where $J^\ast$ is the optimal function of $x$ and $t$ for the optimally combined  incentive $u^{\ast}$, given by
		\begin{equation}\label{eq17}
			\begin{aligned}
				J^\ast=\frac{1}{2}\int^{t_{f}}_{t_0}\|G(x, u^{\ast}, t)\|_2^2\textrm{d}t=\frac{1}{2}\int^{t_{f}}_{t_0}[G(x, u^{\ast}, t)]^2\textrm{d}t.
			\end{aligned}
		\end{equation}
		Solving $\frac{\partial H}{\partial u}=0$, one can check that the optimally combined  incentive protocol  $u^{\ast}$ satisfies
		\begin{equation}\label{eq18}
			\begin{aligned}
				u^{\ast}=-\frac{\omega\frac{2-\alpha+(\alpha-\beta)p}{4}x(1-x)}{n^2(n-1)^2[px+(1-p)(1-x)]^2} \frac{\partial J^{\ast}}{\partial x}.
			\end{aligned}
		\end{equation}
		The standard approach is to solve the canonical equations of \eqref{eq16} to obtain the optimal incentive protocol.  However, since the system \eqref{eq8} is nonlinear, it is intractable to obtain the optimal incentive protocol by a direct calculation.  
		To this end, we employ  the approach of HJB equation to solve the optimal control problem \eqref{eq10}, and this equation can be written as
		\begin{equation}\label{eq19}
			\begin{aligned}
				-\frac{\partial J^\ast}{\partial t}&=H(x, u^{\ast}, t)=\frac{1}{2}[G(x, u^*, t)]^2+\frac{\partial J^{\ast}}{\partial x}F(x, u^*,  t).
			\end{aligned}
		\end{equation}
		By substituting \eqref{eq18} into the HJB equation, one gets that
		\begin{equation}\label{eq20}
			\begin{aligned}
				-\frac{\partial J^{\ast}}{\partial t}&=\frac{1}{2}\big\{n(n-1)u^* [px+(1-p)(1-x)] \big\}^2\\
				&+\frac{\partial J^{\ast}}{\partial x}\frac{\omega x(1-x)}{4}\big\{(\alpha-\beta)bx-(2-\beta)c\\
				&+[2-\alpha+(\alpha-\beta)p]u^*\big\}.
			\end{aligned}
		\end{equation}
		Under Assumption \ref{assumption1}, one knows that the optimal function $J^{\ast}$ is independent of $t$, and then \eqref{eq20} turns out to be 
		\begin{equation}\label{eq21}
			\begin{aligned}
				\frac{\partial J^\ast}{\partial t}=0.
			\end{aligned}
		\end{equation}
		Furthermore, it is derived that 
		$\frac{\partial J^{\ast}}{\partial x}=0$ 
		or
		\begin{equation}\label{eq22}
			\begin{aligned}
				\frac{\partial J^{\ast}}{\partial x}&=\frac{8n^2(n-1)^2[px+(1-p)(1-x)]^2}{\omega [2-\alpha+(\alpha-\beta)p]^2x(1-x)}\\
				&\times[(\alpha-\beta)bx-(2-\beta)c].
			\end{aligned}
		\end{equation}
		According to Remark \ref{Remark4},  we have $(2-\beta)c>(\alpha-\beta)b$, implying  that $(\alpha-\beta)bx<(2-\beta)c$ for all $x\in(0, 1)$. Thus it holds that $\frac{\partial J^{\ast}}{\partial x}<0$ in Eq. \eqref{eq22}. In view of the fact that  $u>0$, $\frac{\partial J^{\ast}}{\partial x}$ is always negative (i.e., $\frac{\partial J^{\ast}}{\partial x}<0$). Thus, \eqref{eq22}  can be satisfied, and by substituting it into \eqref{eq20}, one gets
		the optimally combined incentive protocol $u^{\ast}$ as
		\begin{equation}\label{eq23}
			\begin{aligned}
				u^{\ast}=\frac{2[(2-\beta)c-(\alpha-\beta)bx]}{2-\alpha+(\alpha-\beta)p}. 
			\end{aligned}
		\end{equation}
		Furthermore, the system equation \eqref{eq7} can be rewritten as
		\begin{equation}\label{eq24}
			\begin{aligned}
				\dot{x}(t)=\omega\frac{(2-\beta)c-(\alpha-\beta)bx}{4}x(1-x).
			\end{aligned}
		\end{equation}
		By checking \eqref{eq24}, one can see that the system needs infinitely long time to reach a full cooperation state from any initial  state $x_0\in (0, 1)$.  Instead, under 
		Assumption \ref{assumption3}, the terminal state $x(t_{f})$ is assumed to be $1-\delta$, where $\delta$ is the parameter that determines the  proportion of cooperators in the population at the terminal time
		$t_{f}$. Since $x(t)$ increases monotonically over time $t$,  $x(t_{f})> x_{0}$ (i.e., $x_0+\delta<1$). 
		Combining \eqref{eq23} and \eqref{eq24},  the amount of cumulative  incentive required by $u^{\ast}$ for the system to reach the expected terminal state $x(t_{f})$ from the initial state $x_0$ is
		\begin{equation}\label{eq25}
			\begin{aligned}
				&J^\ast=\frac{1}{2}\int^{t_{f}}_{t_0}\|G(x, u^{\ast}, t)\|_2^2\textrm{d}t\\
				&=\frac{1}{2}\int^{t_{f}}_{t_0}\big\{n(n-1)u^* [px+(1-p)(1-x)] \big\}^2\textrm{d}t\\
				&=\frac{1}{2}\int^{x(t_{f})}_{x(t_0)}\frac{n^2(n-1)^2(u^*)^2 [px+(1-p)(1-x)]^2}{\omega\frac{(2-\beta)c-(\alpha-\beta)bx}{4}x(1-x)}\textrm{d}x\\
				&=\frac{8n^2(n-1)^2}{\omega [2-\alpha+(\alpha-\beta)p]^2}\int^{x(t_{f})}_{x(t_0)}\frac{1}{x(1-x)}\\
				&\times[(2-\beta)c-(\alpha-\beta)bx] [px+(1-p)(1-x)]^2\textrm{d}x\\
				&=\frac{8n^2(n-1)^2}{\omega [2-\alpha+(\alpha-\beta)p]^2} \Big\{(2-\beta)c\int^{1-\delta}_{x_0} \Big[p^2\frac{x}{1-x}\\
				&+2p(1-p)+(1-p)^2\frac{1-x}{x}\Big] \textrm{d}x-(\alpha-\beta)b  \\
				&\int^{1-\delta}_{x_0}\Big[p^2\frac{x^2}{1-x}+2p(1-p)x+(1-p)^2(1-x)\Big]\textrm{d}x\Big\}\\
				&=\frac{\phi [(2-\beta)c\mathcal{A}-(\alpha-\beta)b\mathcal{B}]}{[2-\alpha+(\alpha-\beta)p]^2},
			\end{aligned}
		\end{equation}
		where $\phi=\frac{8n^2(n-1)^2}{\omega}$, $\mathcal{A}=\int^{1-\delta}_{x_0} \big[p^2\frac{x}{1-x}+2p(1-p)+(1-p)^2\frac{1-x}{x}\big] \textrm{d}x=(2p-1)^2(x_{0}-1+\delta)+p^2\ln\big(\frac{1-x_0}{\delta}\big)+(1-p)^2\ln\big(\frac{1-\delta}{x_0}\big)$, and $\mathcal{B}=\int^{1-\delta}_{x_0}\big[p^2\frac{x^2}{1-x}+2p(1-p)x+(1-p)^2(1-x)\big]\textrm{d}x=(1-2p)(1-x_{0}-\delta)-\frac{(2p-1)^2}{2}(1-x_{0}-\delta)(1+x_{0}-\delta)+p^2\ln\big(\frac{1-x_0}{\delta}\big)$.
	\end{proof}
	
	\begin{figure}[!t]
		\begin{center}
			\includegraphics[width=3.7in]{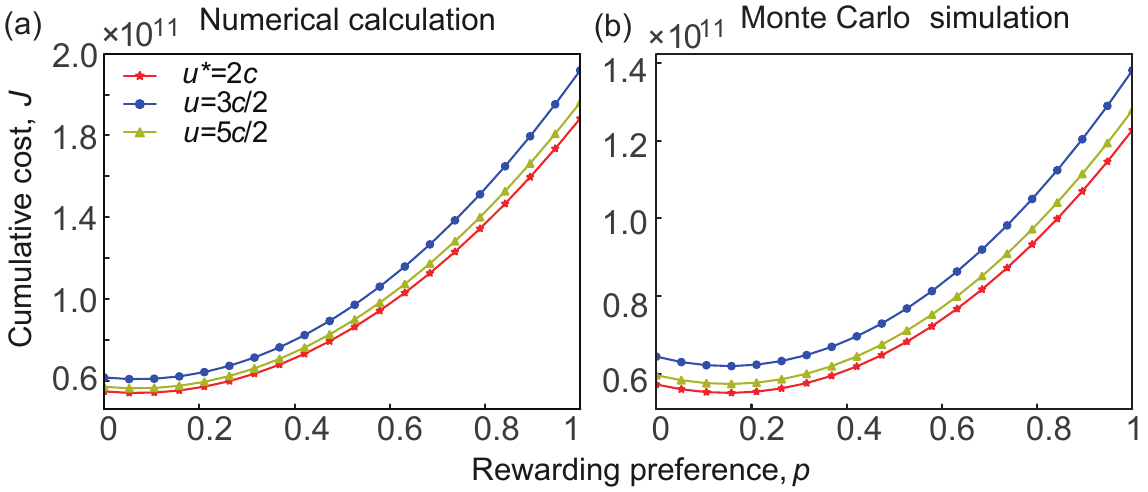}
			\caption{The required amount of cumulative cost to reach the expected proportion of cooperators in the population for the optimally combined incentive in dependence of the rewarding perference $p$ for different incentive protocols, including $u^*$,  $u=\frac{3c}{2}$, and $u=\frac{5c}{2}$. Besides, panel $(a)$  shows the results derived from numerical calculations based on \eqref{eq25}, and panel $(b)$ presents the results obtained from Monte carlo simulations by averaging over $200$ independent simulation runs.
				Parameters: $n=100$, $x_0=0.15$, $\delta=0.01$, $\alpha=\beta=0.5$, $\omega=0.01$, $b=2$, and $c=1$.}\label{fig5}
		\end{center}
	\end{figure}

		\begin{figure}[!t]
		\begin{center}
			\includegraphics[width=3.5in]{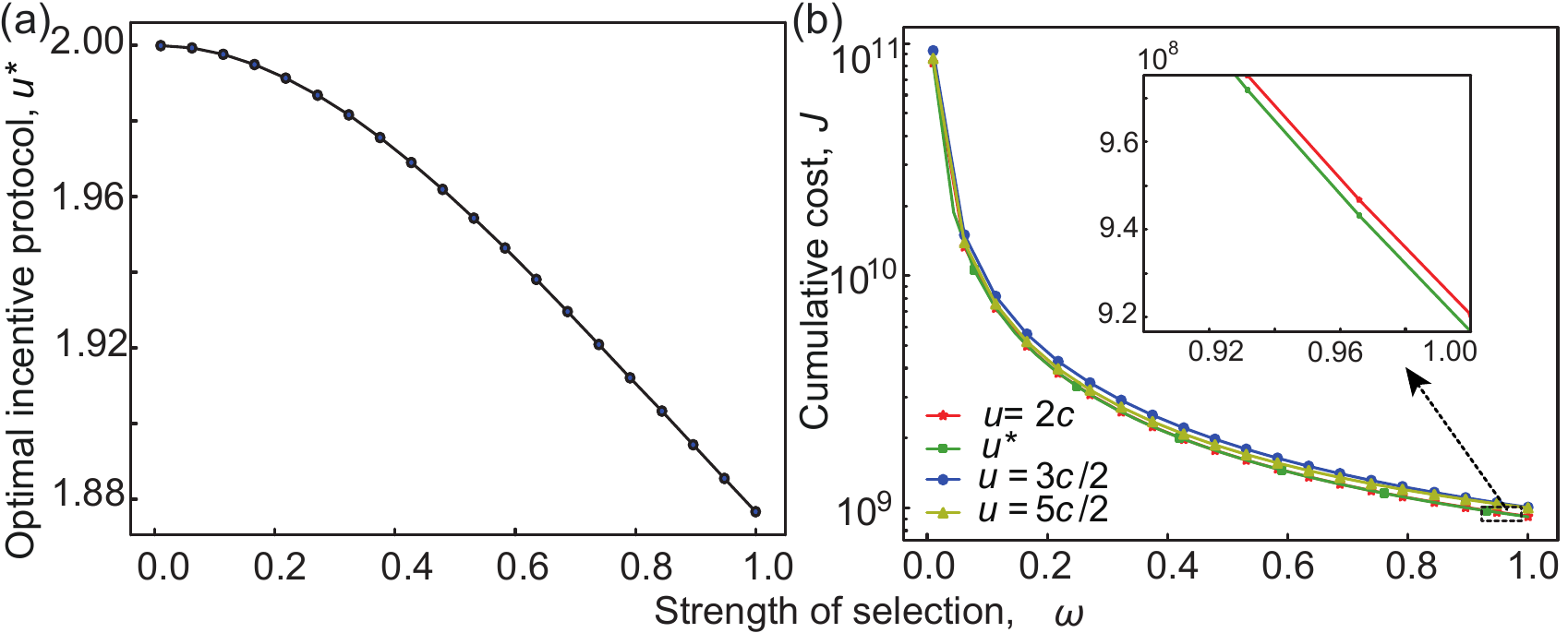}
			\caption{The optimal incentive prorocol $u^*$ and the cumulative cost $J$ in dependence of the strength of selection $\omega$.	Here, we also consider four different incentive protocols in panel $(b)$, including $u^*$, $u=2c$, $u=\frac{3c}{2}$, and $u=\frac{5c}{2}$. The values of $\alpha$ and $\beta$ are $0.5$ (i.e., $\alpha=\beta=0.5$). Parameters: $n=100$, $x_0=0.15$, $\delta=0.01$,  $p=0.5$, $b=2$, and $c=1$.}\label{fig6}
		\end{center}
	\end{figure}

	Furthermore, let $\alpha=\beta=0$, and  we accordingly
	give an remark to elaborate on the special case without error.
	\begin{Remark}\label{Remark5}
		When institutional errors are not considered (i.e., $\alpha=\beta=0$), the obtained optimally combined incentive protocols become $u^{\ast}=2c$, and the amounts of cumulative costs are given by $J^\ast|_{\alpha=0, \beta=0}=\frac{c\phi \mathcal{A}}{2}=\frac{4n^2(n-1)^2 c}{\omega }  [(2p-1)^2(x_{0}-1+\delta)+p^2\ln\big(\frac{1-x_0}{\delta}\big)+(1-p)^2\ln\big(\frac{1-\delta}{x_0}\big)]$.
	\end{Remark}

Combining Theorem~\ref{thm1} and Remark~\ref{Remark4},   one knows that all obtained optimal incentive protocols are $u^{\ast}=\frac{2[(2-\beta)c-(\alpha-\beta)bx]}{2-\alpha+(\alpha-\beta)p}$, and these incentive protocols can make the dynamical system converge to the full cooperation state (i.e., $x=1$) for all $p\in[0, 1]$. However, it can be observed that the requested cumulative  cost is determined by the rewarding preference $p$,  the error probabilities (i.e., $\alpha$ and $\beta$), the initial cooperation level $x_0$, and the difference $\delta$ between the full cooperation state (i.e., $x=1$) and the desired cooperation state (i.e., $x=1-\delta$).  To this end, we will provide comparative analysis regarding the cumulative costs in presence of payoff-observation errors
	with those in error-free cases, and further to determine whether the presence of  error significantly impacts incentive costs compared to the cases without error. This comparison is presented in the next subsection.

	\begin{figure}[!t]
		\begin{center}
			\includegraphics[width=3.6in]{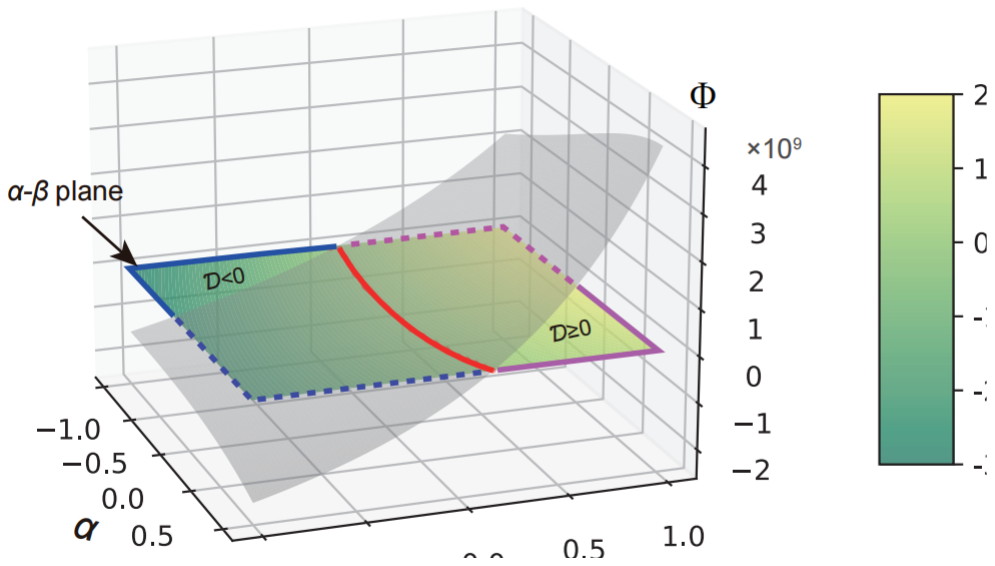}
			\caption{The optimal cost difference $\Psi$ between the cases with and without payoff-observation errors.  The colored surface illustrates the contour plane of the function  ${D}$ with respect to $\alpha\in[-1,1]$ and $\beta\in[-1,1]$ on the $\alpha-\beta$ plane. This contour plane is divided by two regions: the blue-red line region represents the case of ${D}<0$, with $(\alpha, \beta)\in \mathcal{L}_1$ and $\mathcal{L}_1=\{(\alpha, \beta)|\mathcal{D}<0, \alpha\in[-1,1],  \textrm{and} \,\,\, \beta\in[-1,1]\}$.   The pink-red line region  corresponds to  the other case of ${D}\geq0$, with $(\alpha, \beta)\in\mathcal{L}_2$ and $\mathcal{L}_2=\{(\alpha, \beta)|\mathcal{D}\geq 0, \alpha\in[-1,1], \textrm{and} \,\,\, \beta\in[-1,1]\}$. Besides, the gray shadowed surface presents the function values of $\Psi$ for all $\alpha\in[-1,1]$ and $\beta\in[-1,1]$. Parameters: $n=100$, $p=0.5$, $x_0=0.1$, $\delta=0.1$, $\omega=0.1$, $b=2$, and $c=1$.}\label{fig7}
		\end{center}
	\end{figure}

	\subsection{The impact of errors on the cumulative cost induced by the optimal incentives}
	
	In this section, we now proceed to discussing how the error probabilities (i.e., $\alpha$ and $\beta$) affect the cumulative cost of $J^*$ induced by the optimally combined incentives. For notational simplicity, we define the cumulative cost difference between the optimal incentive protocols in a gaming environment with payoff-observation error and the one without such error as 
	the function $\Psi$, which from Remark \ref{Remark4} and Theorem \ref{thm1} is given by 
	\begin{equation}
		\begin{aligned}
			&\Psi=J^\ast-J^\ast|_{\alpha=0, \beta=0}\\
			&=\frac{\phi [(2-\beta)c\mathcal{A}-(\alpha-\beta)b\mathcal{B}]}{[2-\alpha+(\alpha-\beta)p]^2}-\frac{\phi c\mathcal{A}}{2}\\
			&=\phi \frac{2 [(2-\beta)c\mathcal{A}-(\alpha-\beta)b\mathcal{B}]-c\mathcal{A} [2-\alpha+(\alpha-\beta)p]^2 }{2[2-\alpha+(\alpha-\beta)p]^2}\\
			&=c\mathcal{A}\phi \frac{\mathcal{D}}{2[2-\alpha+(\alpha-\beta)p]^2},
		\end{aligned}
	\end{equation}
	where $\mathcal{D}=2(2-\beta)-[2-\alpha+(\alpha-\beta)p]^2 -2(\alpha-\beta)\frac{b\mathcal{B}}{c\mathcal{A}}$. %For the convenience of analysis, we define $\mathcal{L}_1=\{(\alpha, \beta)|\mathcal{D}<0, \alpha\in[-1,1],  \textrm{and} \,\,\, \beta\in[-1,1]\}$, and $\mathcal{L}_2=\{(\alpha, \beta)|\mathcal{D}\geq 0, \alpha\in[-1,1], \textrm{and} \,\,\, \beta\in[-1,1]\}$. 

	\subsubsection{Comparison analysis of cumulative costs between scenarios with and without payoff-observation errors}
	We aim to study the relationship between the cumulative cost in presence of payoff-observation errors and the one in error-free cases in the following theorem. 
	\begin{theorem}\label{thm2}
		The following statements hold:
		\begin{enumerate}
			\item If  $\mathcal{D}<0$ (i.e., $2(2-\beta)-[2-\alpha+(\alpha-\beta)p]^2 <2(\alpha-\beta)\frac{b\mathcal{B}}{c\mathcal{A}}$), then $\Psi<0$ (i.e., $J^*<J^\ast|_{\alpha=0, \beta=0}$).
			\item If $\mathcal{D}\geq 0$ (i.e., $2(2-\beta)-[2-\alpha+(\alpha-\beta)p]^2 \geq 2(\alpha-\beta)\frac{b\mathcal{B}}{c\mathcal{A}}$), then $\Psi\geq 0$  (i.e., $J^*\geq J^\ast|_{\alpha=0, \beta=0}$).
		\end{enumerate}
	\end{theorem}
	\begin{proof} 
		Since $\mathcal{A}$, $c$, $\phi$, and $[2-\alpha+(\alpha-\beta)p]^2$ are positive (i.e., $\mathcal{A}>0$, $c>0$, $\phi>0$, and $[2-\alpha+(\alpha-\beta)p]^2>0$), one can check that $\Psi<0$ if $\mathcal{D}<0$.
		While $\Psi\geq 0$ if $\mathcal{D} \geq 0$, that is, $\Psi\geq 0$.
	\end{proof}
	
	One interpretation of Theorem \ref{thm2} is  that under the condition of $\mathcal{D}<0$, the cost optimal difference $\Psi$ is always negative (i.e., $\Psi<0$ and $J^*<J^\ast|_{\alpha=0, \beta=0}$), implying that 
	applying the optimal incentive in the presence of payoff-observation errors always leads to a lower cumulative cost, when compared to the scenario without such errors. Otherwise,  $\Psi$ is always postive (i.e., $\Psi\geq0$ and $J^*\geq J^\ast|_{\alpha=0, \beta=0}$), suggesting that applying the optimal incentive in the absence of error induces a lower cost.
	
	\begin{algorithm}[htb]
		\caption{Projected Gradient Descent with Normalization}
		\label{algorithm1}
		\begin{algorithmic}[1] %这个1 表示每一行都显示数字
			\REQUIRE ~~\\ %算法的输入参数：Input
			Learning rate $\eta$, 
			maximum iterations $N$, and objective function given by Eq.~(27).
			\STATE Randomly initialize: \\    
			$\alpha\gets randClosed (-1, 1)$, $\beta \gets randClosed (-1, 1)$
			$p \gets randClosed (0, 1)$, 
			$x_0 \gets randOpen (0, 1)$, and
			$\delta \gets randOpen (0, 1-x_0)$,\\
			where $randClosed$ denotes the random selection of a value from the specified closed interval, while $randOpen$ refers to a randomly chosen value from the specified open interval. Repeat step 1 until $\mathcal{D}<0$ holds.
			\STATE \textbf{For} $i = 1$ to $N$:
			\begin{enumerate}
				\item Compute gradient normalization $\nabla_{L}$:
				\[
				\nabla \gets \text{ComputeGradients}(\Psi, \boldsymbol{s}) \,\,\textrm{and}\,\, \nabla_{L} = \frac{\nabla}{\left\| \nabla \right\|_2},
				\]
				where \texttt{ComputeGradients} calculates the partial derivatives of $\Psi$ with respect to $\boldsymbol{s}\in\bar{S}$. Gradient normalization plays a crucial role in balancing the partial derivatives of $\Psi$, thereby mitigating the risk of the algorithm exceeding the computable numerical range~\cite{Chen2018ICML}.
				\item Update parameter:
				\[
				\boldsymbol{s} \gets \boldsymbol{s} - \eta \nabla_{L}.
				\]
				\item Project parameters to valid ranges:
				\[\boldsymbol{s}\gets\arg min_{\boldsymbol{z}\in \bar{S}} \left\| \boldsymbol{s}-\boldsymbol{z} \right\|^2_{2}, 
				\]
				which is Euclidean projection onto $\bar{S}$ to ensure that the constraint of $\bar{S}$ is satisfied. 
			\end{enumerate}
			\STATE Compute final value:
			\[
			\Psi^* \gets \Psi(\boldsymbol{s}).
			\]
			
			\ENSURE ~~\\ %算法的输出：Output
			Optimal parameters $\boldsymbol{s^*}=(\alpha^*, \beta^*, p^*, x_0^*, \delta^*)$ and the minimum value $\Psi^*$. 
		\end{algorithmic}
	\end{algorithm}

	\subsubsection{Optimization of the cost difference between the cases with and without payoff-observation errors}  From Theorem \ref{thm2}, it can be seen that when $\mathcal{D}<0$, the cost difference $\Psi$ between the scenario with errors and the scenarios without error is positive, meaning that the cumulative cost with error is lower than the one without error. Therefore, under this condition, we further aim to explore when this cost difference is minimal. To this end, we formulate an optimization problem, which is described by 
	\begin{equation}
		\begin{aligned}\label{eq27}
			&\min\,\, \Psi\\
			& s.t. \,\, \left\{\begin{array}{lc}
				\mathcal{D}<0,\\
				0\leq p\leq 1,  \\
				0<x_0<1,\\
				0<\delta<1-x_0, \\
				-1\leq\alpha, \beta \leq 1.
			\end{array}\right.
		\end{aligned}
	\end{equation}
	To facilitate the design of the algorithm, here we define the constraints in the optimization problem as a set of all parameters, denoted by 
	\begin{equation}\label{eq28}
		\begin{aligned}
			\bar{S}=\{&\boldsymbol{s}=(\alpha, \beta, p, x_0, \delta)| \mathcal{D}<0, p\in(0, 1), x_0\in(0, 1), \\
			&\delta\in(0,1-x_0), \textrm{and} \,\,\alpha, \beta \in[0,1]\}\subseteq \mathbb{R}^5.
		\end{aligned}
	\end{equation}
	Normally, we need to solve the optimization problem mentioned above to obtain the exact expression of optimal parameters $\boldsymbol{s^*}=(\alpha^*, \beta^*, p^*, x_0^*, \delta^*)$,  such that the cost difference $\Psi$ is minimized. However,  since both  $\Psi$ and $\mathcal{D}$ are nonlinear functions, it is difficult to derive the exact expression of the optimal protocols through theoretical analysis. Instead, we turn to numerical methods to solve this optimization problem and design Algorithm \ref{algorithm1}. After numerical calculation, we derive the numerical values of these parameters as
	$p^*=1.0$, $x_0^*=1 \times 10^{-6}$, $\delta^*=1 \times 10^{-6}$, $\alpha^*= 0.999999\approx1$, and $\beta^*=-0.999999\approx-1$. From this, we can also 
	observe that when $\alpha^*$ approaches $1$ and $\beta^*$ approaches $-1$, $\Psi$ is minimal, this implies that the error parameter of $C$-agent observing a deviation in the payoff of $D$-agent is $1$, leading to that the observed payoff of $D$-agent is $\hat{\Pi}_{D} = (1-\alpha) \Pi_{D}=0$, and thus $C$-agent will continue to maintain its strategy. While $D$-agent observes a deviation of $0$ in the payoff of $C$-agent, resulting in that the observed payoff of $C$-agent is $\hat{\Pi}_{C} = (1-\beta) \Pi_{C} = 2 \Pi_{C}$, and thus $D$-agent will change it strategy since the payoff of opponent is relatively higher. Therefore, 
	the optimal cost in presence of error is significantly lower than the one in error-free scenarios. Therefore, our numerical results are consistent with our intuitive understanding.
	
	%p = 1.0, x = 1e-06, delta = 1e-06, alpha = 0.9999999999999062, beta = -0.9999999999999062
	
	%最小值: -59664408665.93683

	\section{Numerical Results}\label{sec4}
	In this section, we present numerical calculations to verify the analytical results obtained in Lemmas \ref{lem1}-\ref{lem3} and Theorems \ref{thm1}-\ref{thm2} presented in the previous section. The corresponding results are illustrated in Figs.~\ref{fig1}-\ref{fig7}. 
	
	\subsection{Numerical results for the case of gradient of selection}
	We consider three different cases, which are $\alpha=\beta$ (Lemmas \ref{lem1}), $\alpha>\beta$ (Lemmas \ref{lem2}), and $\alpha<\beta$ (Lemmas \ref{lem3}). For each cases, we show the gradient of selection $\dot{x}$ as a function of the proportion of cooperators $x$, and the numerical results are depicted in Figs.~\ref{fig2}-\ref{fig4}, respectively. 
	
	\subsubsection{$\alpha=\beta$} 
	As shown in Fig.~\ref{fig2}, the system exhibits two distinct states depending on the value of 
	$u$.  When  $0\leq u<c$, the gradient of selection remains negative for all the initial cooperation levels $x_0\in(0, 1)$, making the system to converge to a full defection state. Conversely, when $u>c$,  the gradient of selection  is always positive, driving the system towards a full cooperation state. Thus, 
	we verify the theoretical results obtained 
	in Lemma \ref{lem1}.
	
	\subsubsection{$\alpha>\beta$}  
	Fig.~\ref{fig3} reveals a more complex dynamic. When $0\leq u\leq\frac{(2-\beta)c-(\alpha-\beta)b}{2-\alpha+(\alpha-\beta)p}$, the gradient of selection remains negative over the interval $x_0\in(0, 1)$, making the system to converge to a full defection state.  However, for $\frac{(2-\beta)c-(\alpha-\beta)b}{2-\alpha+(\alpha-\beta)p}<u<\frac{(2-\beta)c}{2-\alpha+(\alpha-\beta)p}$,  the intermediate equilibrium point exists. This equilibrium point is unstable
	and acts as a threshold separating two basins of attraction, that is,  if the initial cooperation level is below this threshold,  then the system converges to the  full defection state; otherwise, it reaches the full cooperation state. Notably, both absorbing states are stable, indicating that PDG can indeed be transformed into a coordination game. When $u\geq\frac{(2-\beta)c}{2-\alpha+(\alpha-\beta)p}$ 
	, the intermediate equilibrium point disappears and the gradient of selection remains positive for all $x_0$, ensuring the convergence to a full cooperation state.  Therefore, we verify the theoretical results obtained in Lemma \ref{lem2}.

	\subsubsection{$\alpha<\beta$} 
	Fig.~\ref{fig4} illustrates that when $0\leq u\leq\frac{(2-\beta)c}{2-\alpha+(\alpha-\beta)p}$, 
	the gradient of selection is entirely negative, making the system to the full defection state for any $x_0\in(0, 1)$. In contrast, when $\frac{(2-\beta)c}{2-\alpha+(\alpha-\beta)p}<u<
	\frac{(2-\beta)c-(\alpha-\beta)b}{2-\alpha+(\alpha-\beta)p}$, the 
	intermediate equilibrium point is stable, while the boundary equilibria are unstable. Consequently, the system evolves to a coexistence state for any initial cooperation level $x_0\in(0, 1)$. When $u\geq\frac{(2-\beta)c-(\alpha-\beta)b}{2-\alpha+(\alpha-\beta)p}$, the intermediate equilibrium point disappears and the gradient of selection remains positive for all the initial cooperation level $x_0\in (0, 1)$, leading the system to converge  towards the full cooperation state. Therefore, the above numerical results support the analytical results on complete graphs as stated in 
	Lemma \ref{lem3}.
	
	\subsection{Numerical results for the case of optimal incentive protocol}
	
	In order to verify Theorem \ref{thm1}, we provide some  numerical calculations and Monte Carlo simulation results in Fig.~\ref{fig5}. Here, we also consider that the $\delta$ value is sufficiently small, so that the full cooperation state $x(t_f)$
	could be reached.            
	In Fig.~\ref{fig5}, we show the cumulative cost value  as a function of the rewarding preference $p$ for the
	the optimally combined protocol $u^*$, and two other given incentive protocols $u=\frac{3c}{2}$ and $u=\frac{5c}{2}$. We find that the amount of cumulative cost of the optimally combined protocol is lowest when it is compared with the other two cost values for all the rewarding perference $p$ values.	Furthermore, from panels $(a)$ and $(b)$ of Fig.~\ref{fig5}, one can see that our simulation results coincide with numerical calculation results, both of which support our theoretical results as stated in Theorem \ref{thm1}, and confirm that the obtained protocol $u^*$ is the optimal incentive scheme and dominates other incentive protocols in the cost of implementing incentives.
	
Besides, we further extend our results obtained under
 weak selection ($\omega\rightarrow 0$) to the non-weak selection case ($\omega>0$), as shown in Fig.~\ref{fig6}.  Panels $(a)$ and $(b)$ illustrate the optimal incentive protocol $u^*$ and the cumulative cost $J$ as a function of the strength of selection $\omega$, respectively. We find that both $u^*$ and $J$ decrease monotonically with respect to $\omega$, indicating that moderate or strong selection leads to a lower implementation cost. 
 Moreover, from Fig.~\ref{fig6} (b), although $J$ decreases monotonically with $\omega$, the optimal protocol $u^*$ consistently outperforms the alternative protocols in terms of the cost of implementing incentives.
 Notably, from panel $(b)$ and its the inset subfigure, we see that the cumulative costs induced by the optimal incentive protocol $u^*$ and the weak-selection benchmark optimal protocol $u=2c$ are close,  but the former remains consistently lower than the latter.

	\subsection{Numerical results for the case of minimizing the payoff difference}
	Finally, we provide numerical calculations to validate the theoretical results of the optimal incentives obtained in Theorem \ref{thm2}, and illustrate the optimal cost difference $\Psi$ between the cases with and without payoff-observation errors as a function of the error parameter
	set of $(\alpha, \beta)$, as shown in Fig.~\ref{fig7}.  We find that the cost difference  $\Psi$ are consistently negative in the blue-red line region of ${D}<0$ with $(\alpha, \beta)\in \mathcal{L}_1$, while it is positive in the pink-red line region where ${D}\geq0$ with $(\alpha, \beta)\in \mathcal{L}_2$. These findings highlight the distinct impact of payoff-observation errors depending on whether the error parameter set $(\alpha, \beta)$ within the region of $\mathcal{L}_1$ or $\mathcal{L}_2$.

\section{Conclusion and Discussion}\label{sec5}
	
In this work, we investigate the optimal incentive control schemes in evolutionary games with payoff-observation errors, and also study the impact of this error on the implementation cost of incentives. We begin by deriving the dynamical equations governing the fraction of cooperators using the mean-field theory. Through a detailed analysis of the system equation, there are at most three equilibria: the full defection state, the coexistence state of cooperators and defectors, and the full cooperation state. By further analyzing the stability of these equilibria, we identify the minimal incentive level required to promote the evolution of cooperation in a well-mixed population. We then construct an index function to quantify the total execution cost, and accordingly obtain the optimal incentive protocols inducing to the minimal cumulative cost for the emergence of cooperation. Furthermore, our theoretical and numerical results show that the cumulative costs of optimal incentive protocols in the presence of payoff-observation errors is lower than those in error-free scenarios, and also determine the theoretical conditions for these results. Finally, we formulate an optimization problem to investigate the conditions under which the optimal cost in the presence of errors can be minimized relative to the error-free scenario, and design an algorithm for its the numerical solution. 
Beyond its theoretical implications, our findings may also provide practical insights for real-world systems in which decision-making is affected by perceptual inaccuracies~\cite{Spitzer2025MS}. Potential domains of application include public policy design, market competition, and collective decision-making in social and organizational networks, where both incentive mechanisms and information distortions play pivotal roles~\cite{Ferguson2022cdc, Mann2017PNAS}.

In this work, payoff-observation errors $\alpha$ and $\beta$ are considered as  non-strategic biases, which captures individual perception distortions or cognitive limitations rather than deliberate misreporting. Under this informational structure, these errors arise automatically from the environment, and agents neither engage in an explicit reporting process nor possess the ability to strategically manipulate observed information.
 Consequently, the optimal incentive protocol can be derived without imposing additional constraints such as incentive compatibility or information verifiability. This simplified setting enables a clear examination of the relationship between evolutionary stability and implementation cost, while also defining the scope of the present framework. In more general environments, however, agents may be able to strategically adjust or selectively disclose publicly observed payoff signals. In such cases, the structure and cost of the optimal incentive scheme may change, and the resulting solution is likely to be second-best rather than first-best. Extending the model to incorporate endogenous reporting mechanisms, strategic signal manipulation, and incentive-compatibility constraints therefore constitutes a meaningful direction for future research.

Notably,  our study focuses exclusively on a common type of individuals errors, payoff-observation errors, while in reality, individuals may also exhibit other forms of behavioral inaccuracies.
For example, when  participating in a game, agents may unintentionally adopt a strategy opposite to the one they originally intended to execute, which is commonly referred to as the \emph{strategy-implementation error}~\cite{ChenInterface2015}. Therefore, an interesting direction for future research is to consider the other types of individual errors, and systematically investigate how these errors influence the cost-efficiency of incentives. 
Furthermore, the present study assumes a well-mixed population, which does not fully capture 
the complex structural characteristics typically observed in real-world social systems. In practice, social interactions often take place on diverse network structures, including weighted networks~\cite{WeiHCYB26, ZhuYHCYB26}, heterogeneous networks~\cite{LiuAYHCYB26}, or various topologically diverse networks (e.g., scale-free networks \cite{BarabSCi1999} and small-world networks \cite{WattsNat1998}). Exploring the optimal incentive protocls with minimal cost for  cooperation under payoff-observation errors in structured populations would therefore be a valuable extension of the current work. In addition, microscopic strategy update rules play a crucial role in shaping the evolutionary dynamics of cooperation. While this work incorporates payoff-observation errors into the pairwise-comparison updating process (i.e., the Fermi rule), several other well-established update rules exist, such as aspiration-based updating~\cite{WangTNNLS2025}, as well as death-birth updating and birth-death updating~\cite{Ohtsuk2006Nature}. Extending the analysis to these alternative update rules would provide a more comprehensive understanding of how payoff-observation errors interact with different behavioral update processes and how such interactions influence the design and cost of incentives for promoting cooperation. Finally, the  game we considered is the PDG, whereas many other game paradigms, such as  coordination games and anti-coordination games~\cite{ZhuYHCYB26} also play important roles in modeling social interactions between agents. Exploring optimal incentive design under observation errors across different game structures would further broaden the applicability and generalizability of the proposed framework.

\section*{appendix A}\label{AppendixA}
In this section, we provide a detailed theoretical analysis
to derive the governing dynamical equation defined in Eq. \eqref{eq6}.

As mentioned in the previous section, all $n$ individuals in a well-mixed population opt for pairwise-comparison update rule to revise its strategy. The resulting evolutionary dynamics can be described exactly by a two-dimensional
birth-death process with reflecting states\cite{VanKampen1992}.  That is, at each time step the number $k$ of cooperators increases or decreases by one or remains unchanged. This corresponds to a discrete time Markov
chain with transitional probabilities
\begin{equation}\label{A.1}\tag{A.1}
	\begin{aligned}
		T^{+}(k)=\frac{k}{n}\frac{n-k}{n}\widetilde{W}_{D\rightarrow C}=\frac{k}{n}\frac{n-k}{n}\frac{1}{1+e^{-\omega[\hat{\Pi}_{C}-\Pi_{D}]}},
	\end{aligned}
\end{equation}
\begin{equation}\label{A.2}\tag{A.2}
	\begin{aligned}
		T^{-}(k)=\frac{k}{n}\frac{n-k}{n}\widetilde{W}_{C\rightarrow D}=\frac{k}{n}\frac{n-k}{n}\frac{1}{1+e^{-\omega[\hat{\Pi}_{D}-\Pi_{C}]}},
	\end{aligned}
\end{equation}
and
\begin{equation}\label{A.3}\tag{A.3}
	\begin{aligned}
		T^{0}(k)=1-T^{+}(k)-T^{-}(k),
	\end{aligned}
\end{equation}
where $T^{+}(k)$ denotes  the transition probability to change from state $k$ to $k+1$, $T^{-}(k)$ presents  the transition probability to change from state $k$ to $k-1$, and $T^{0}(k)$ refers
to the probability that the state $k$ remains unchanged.

Denoting the probability that the system is in state $k$ (i.e., having $k$ cooperators in the system) at time $\tau$ as $P^{\tau}(k)$, one can write the master equation \cite{VanKampen1992} as 
\begin{equation}\label{A.4}\tag{A.4}
	\begin{aligned}
		&P^{\tau+1}(k)-P^{\tau}(k)=P^{\tau}(k-1)T^{+}(k-1)\\
		&+P^{\tau}(k+1)T^{-}(k+1)-P^{\tau}(k)T^{-}(k)\\
		&-P^{\tau}(k)T^{+}(k).
	\end{aligned}
\end{equation}
Now, we introduce the notations $x=\frac{k}{n}$,  $t=\frac{\tau}{n}$, and the probability density $\rho(x, t)=n^2P^{\tau}(k)$. For the transition probabilities, we replace $T^{\pm}(k)\to T^{\pm}(x)$. This yields 
\begin{equation}\label{A.5}\tag{A.5}
	\begin{aligned}
		&\rho(x, t+\frac{1}{n})-\rho(x, t)=\rho(x-\frac{1}{n}, t)T^{+}(x-\frac{1}{n})\\
		&+\rho(x+\frac{1}{n}, t)T^{-}(x+\frac{1}{n})-\rho(x, t)T^{-}(x)\\
		&-\rho(x, t)T^{+}(x).
	\end{aligned}
\end{equation}
For $n\gg1$, the probability densities and the transition probabilities are expanded in a Taylor series at $x$ and $t$. More specific, we have  
\begin{equation}\label{A.6}\tag{A.6}
	\begin{aligned}
		\rho(x, t+\frac{1}{n})\approx\rho(x, t)+\frac{\partial \rho(x, t)}{\partial x}\frac{1}{n},
	\end{aligned}
\end{equation}
\begin{equation}\label{A.7}\tag{A.7}
	\begin{aligned}
		\rho(x\pm\frac{1}{n}, t)\approx\rho(x, t)\pm\frac{\partial \rho(x, t)}{\partial x}\frac{1}{n}+\frac{\partial^2 \rho(x, t)}{\partial x^2}\frac{1}{2n^2},
	\end{aligned}
\end{equation}
and 
\begin{equation}\label{A.8}\tag{A.8}
	\begin{aligned}
		T^{\pm}(x\pm\frac{1}{n})\approx T^{\pm}(x)\pm\frac{\partial T^{\pm}(x)}{\partial x}\frac{1}{n}+\frac{\partial^2 T^{\pm}(x)}{\partial x^2}\frac{1}{2n^2}.
	\end{aligned}
\end{equation}
Let us now look at the terms depending on their order in $\frac{1}{n}$. The terms independent of $\frac{1}{n}$ cancel on both sides of \eqref{A.5}. The first non-vanishing term is of order $\frac{1}{n}$. On the left hand side, we have the term $\frac{\partial \rho(x, t)}{\partial t}$ and on the right hand side, we have
\begin{equation}\label{A.9}\tag{A.9}
	\begin{aligned}
		&-\rho(x, t)\frac{\partial T^{+}(x)}{\partial x}+\rho(x, t)\frac{\partial T^{-}(x)}{\partial x}-T^{+}(x)\frac{\partial \rho(x, t)}{\partial x}\\
		&+T^{-}(x)\frac{\partial \rho(x, t)}{\partial x}=
		\frac{\partial [T^{+}(x)-T^{-}(x)]\rho(x, t)}{\partial x}.
	\end{aligned}
\end{equation}
This term describes the average motion of the system. In physics, it is called the drift term but in biology, it is referred to as selection term. Then, we consider the terms of the order $\frac{1}{n^2}$. On the right hand side, we have
\begin{equation}\label{A.10}\tag{A.10}
	\begin{aligned}
		&\frac{\partial \rho(x, t)}{\partial x} \frac{\partial T^{+}(x)}{\partial x}+\frac{\rho(x, t)}{2}\frac{\partial^2 T^{+}(x)}{\partial x^2}+\frac{T^{+}(x)}{2}\frac{\partial^2 \rho(x, t)}{\partial x^2}\\
		&+\frac{\partial \rho(x, t)}{\partial x}\frac{\partial T^{-}(x)}{\partial x}+\frac{\rho(x, t)}{2}\frac{\partial^2 T^{-}(x)}{\partial x^2}+\frac{T^{-}(x)}{2}\frac{\partial^2 \rho(x, t)}{\partial x^2}\\
		&=\frac{1}{2}
		\frac{\partial^2 [T^{+}(x)-T^{-}(x)]\rho(x, t)}{\partial x^2}.
	\end{aligned}
\end{equation}
This second term, called diffusion in physics, leads to a widening of the probability distribution in the course of time. In biology, it is called genetic or neutral drift, which can be a source of confusion. In the following, higher order terms will be neglected. Thus, we can approximate \eqref{A.5} by
\begin{equation}\label{A.10}\tag{A.10}
	\begin{aligned}
		&\frac{\partial \rho(x, t)}{\partial t}=-\frac{\partial [a(x) \rho(x, t)]}{\partial x}+\frac{1}{2}\frac{\partial^2 [b^2(x)\rho(x, t)]}{\partial x^2},
	\end{aligned}
\end{equation}
where $a(x)=T^{+}(x)-T^{-}(x)$ and $b(x)=\sqrt{\frac{T^{+}(x)-T^{-}(x)}{n}}$. This is the Fokker-Planck equation of the system, describing the deterministic time evolution of
a probability distribution. Equivalently, one can describe the process by a stochastic differential equation that generates a single trajectory. If the noise is microscopically uncorrelated, as in our case, the It$\hat{o}$ calculus has to be applied. In this framework,
the Fokker-Planck equation above corresponds to the Langevin equations 
\begin{equation}\label{A.11}\tag{A.11}
	\begin{aligned}
		\frac{\textrm{d}x(t)}{\textrm{d}t}=a(x)+b(x)\xi,
	\end{aligned}
\end{equation}
where $\xi$ is uncorrelated Gaussian white noise. As $n\rightarrow \infty$, the stochastic term vanishes that allows us to write the infinite approximation as 
\begin{equation}\label{A.12}\tag{A.12}
	\begin{aligned}
		&\frac{\textrm{d}x(t)}{\textrm{d}t}=\lim_{n\rightarrow \infty}a(x)+b(x)\xi=T^{+}(x)-T^{-}(x)\\
		&=x(1-x)\frac{1}{1+e^{-\omega[\hat{\Pi}_{C}-\Pi_{D}]}}-x(1-x)\frac{1}{1+e^{-\omega[\hat{\Pi}_{D}-\Pi_{C}]}}\Big]\\
		&=x(1-x)\Big[\frac{1}{1+e^{-\omega[\hat{\Pi}_{C}-\Pi_{D}]}}-\frac{1}{1+e^{-\omega[\hat{\Pi}_{D}-\Pi_{C}]}}\Big]\\
		&=x(1-x)\Big[\frac{1}{1+e^{-\omega[(1- \beta)\Pi_{C}-\Pi_{D}]}}-\frac{1}{1+e^{-\omega[(1-\alpha)\Pi_{D}-\Pi_{C}]}}\Big],
	\end{aligned}
\end{equation}
where  $\Pi_{C}=(b-c+pu_{R})x-(c-pu_{R})(1-x)$ and $\Pi_{D}=\big[b-(1-p)u_{P}\big]x-(1-p)u_{P}(1-x)$ from Eqs. \eqref{eq2} and \eqref{eq3} when $n\gg1$.
This is how we obtain the deterministic equation \eqref{eq6} in the main text.

\begin{IEEEbiography}[{\includegraphics[width=1in,height=1.25in,clip,keepaspectratio]{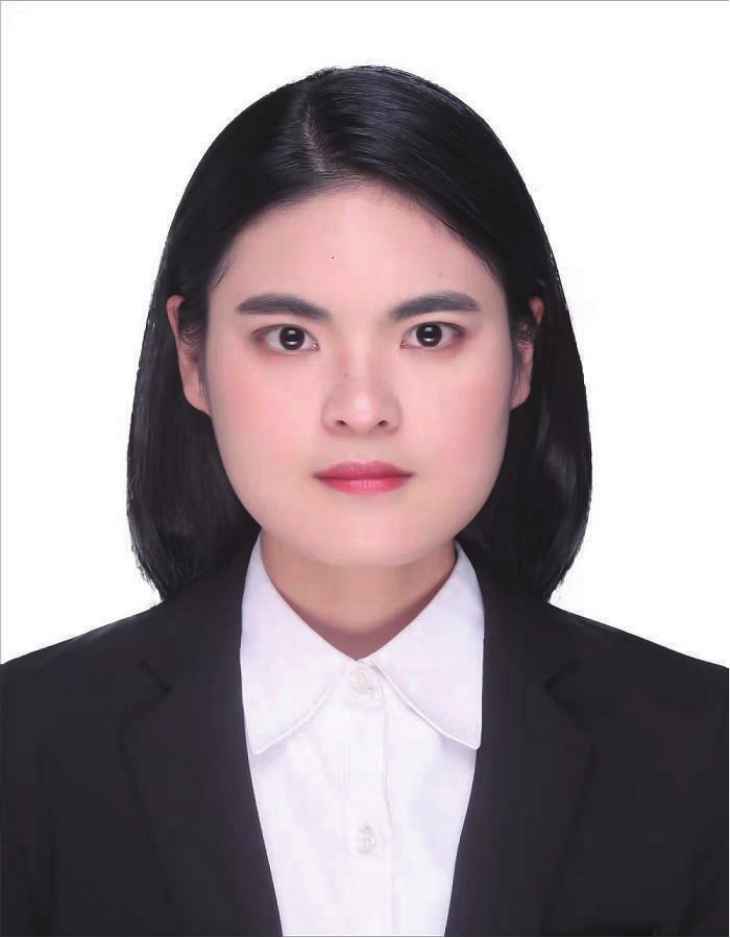}}]{Shengxian Wang} received the Ph.D. degree in mathematics from University of Electronic Science and Technology of China, Chengdu, China, in 2023. From December 2020 to December 2022, Dr. Wang was Guest Ph.D. in University of Groningen, Groningen, The Netherlands. She currently works with Anhui Normal University, Wuhu, China. Her research interests include evolutionary game dynamics, decision-making in game interactions, game-theoretical control, and collective intelligence. 
\end{IEEEbiography}

\begin{IEEEbiography}[{\includegraphics[width=1in,height=1.25in,clip,keepaspectratio]{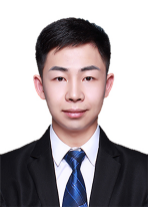}}]{Chengyu Yin} is currently works with Anhui Normal University, Wuhu, China. He received his master from the college of Computer Science and Technology at the Nanjing University of Aeronautics and Astronautics (NUAA) in April 2021; and his bachelor from the college of Computer Science at the Chongqing University (CQU) in July 2018. From October 2019 to June 2020, he worked as a research intern at the JINGDONG AI Research, Nanjing, China. His research interests include spatio-temporal data mining, recommendation system, collective intelligent. 
\end{IEEEbiography}

\begin{IEEEbiography}[{\includegraphics[width=1in,height=1.25in,clip,keepaspectratio]{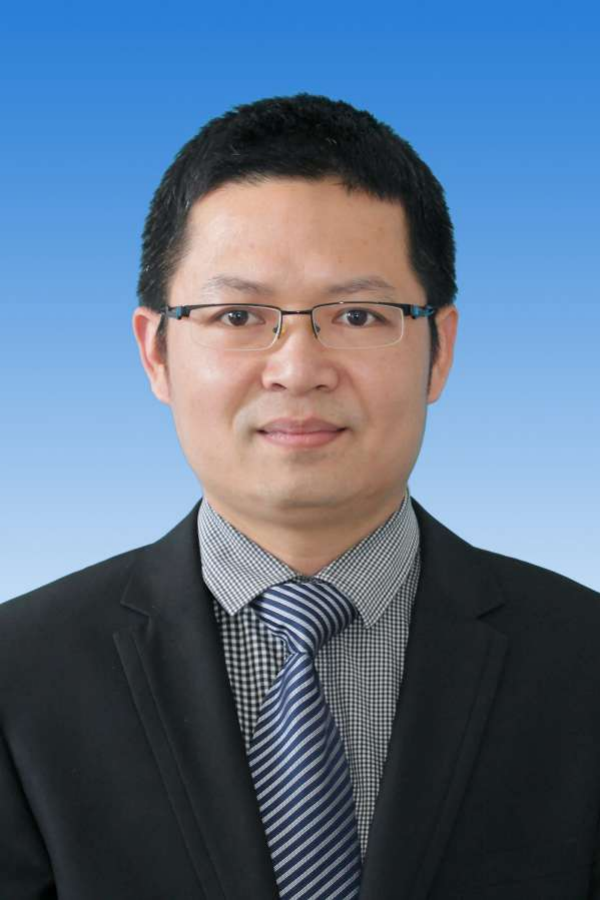}}]{Xiaojie Chen} is currently a professor at School of Mathematical Sciences in University of Electronic Science and Technology of China, China. He received the Bachelor degree in 2005 from National University of Defense Technology, China, and the Ph.D. degree in 2011 from Peking University, China. From September 2008 to September 2009, he was a visiting scholar in University of British Columbia, Canada. From February 2011 to January 2013, he was a postdoctoral research scholar at the International Institute for Applied Systems Analysis (IIASA), Austria. From February 2013 to January 2014, he was a research scholar at IIASA, Austria. He severs as the editorial board member for several international journals. His main research interests include evolutionary game dynamics, decision-making in game interactions, game-theoretical control, and collective intelligence. He has published over $100$ journal papers.
\end{IEEEbiography}

\begin{IEEEbiography}[{\includegraphics[width=1in,height=1.25in,clip,keepaspectratio]{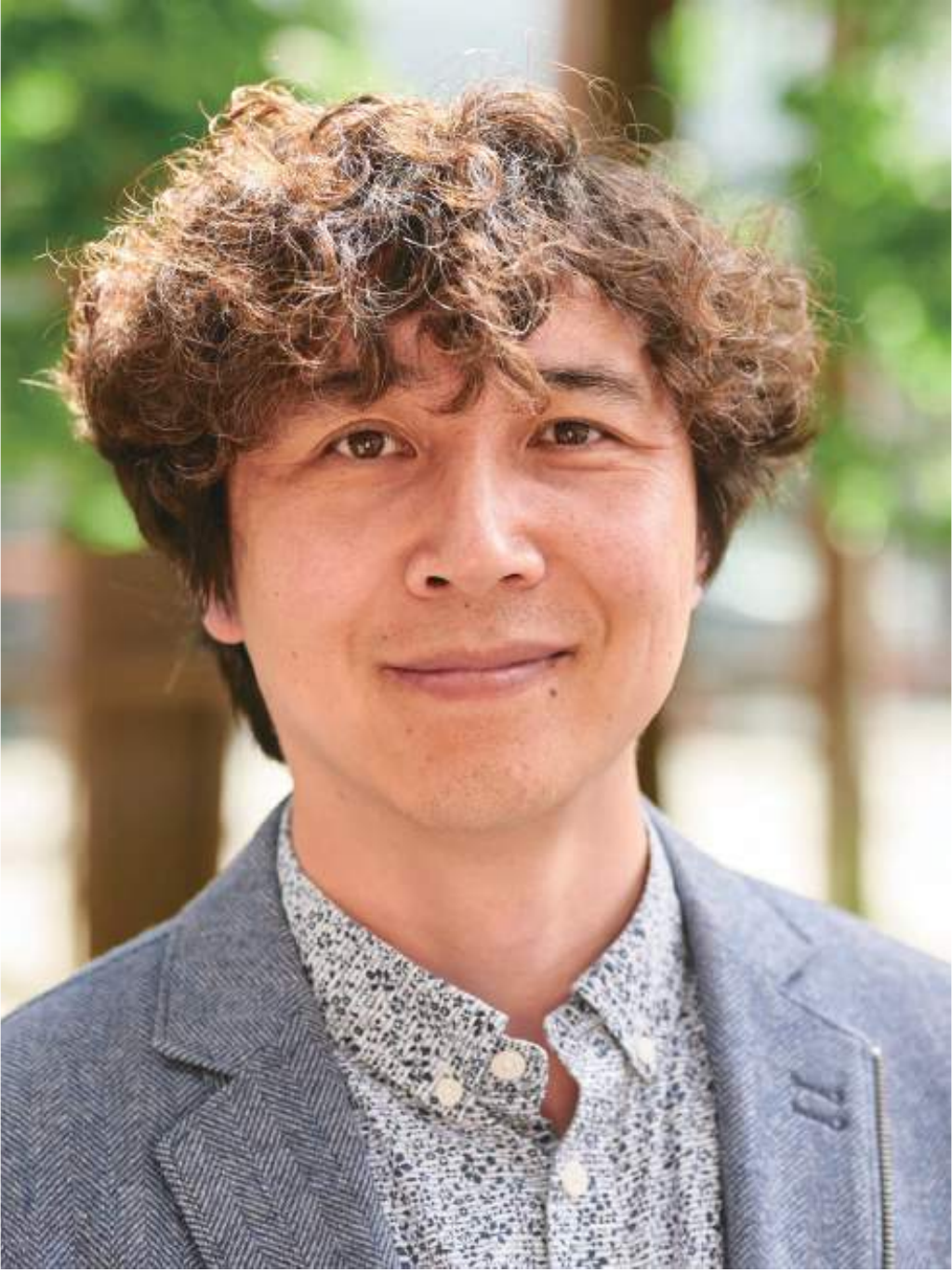}}]{Ming Cao} (Fellow, IEEE) has since 2016 been a professor of networks and robotics with the Engineering and Technology Institute (ENTEG) at the University of Groningen, the Netherlands, where he started as an assistant professor in 2008. Since 2022 he is the director of the Jantina Tammes School of Digital Society, Technology and AI at the same university. He received the Bachelor degree in 1999 and the Master degree in 2002 from Tsinghua University, China, and the Ph.D. degree in 2007 from Yale University, USA. From 2007 to 2008, he was a Research Associate at Princeton University, USA. He worked as a research intern in 2006 at the IBM T. J. Watson Research Center, USA. He is the 2017 and inaugural recipient of the Manfred Thoma medal from the International Federation of Automatic Control (IFAC) and the 2016 recipient of the European Control Award sponsored by the European Control Association (EUCA). He is an IEEE fellow. He is a Senior Editor for Systems and Control Letters, an Associate Editor for IEEE Transactions on Automatic Control, IEEE Transaction of Control of Network Systems and IEEE Robotics \& Automation Magazine, and was an associate editor for IEEE Transactions on Circuits and Systems and IEEE Circuits and Systems Magazine. He is a member of the IFAC Council and a vice chair of the IFAC Technical Committee on Large-Scale Complex Systems. His research interests include autonomous robots and multi-agent systems, complex networks and decision-making processes.
\end{IEEEbiography}


\begin{thebibliography}{1}
\bibliographystyle{IEEEtran}



%\bibitem{ref7}
%A. Levant, ``Exact differentiation of signals with unbounded higher derivatives,''  in \textit{Proc. 45th IEEE Conf. Decis.
%Control}, San Diego, CA, USA, 2006, pp. 5585--5590. DOI: 10.1109/CDC.2006.377165.

%\bibitem{ref8}
%M. Fliess, C. Join, and H. Sira-Ramirez, ``Non-linear estimation is easy,'' \textit{Int. J. Model., Ident. Control}, vol. 4, no. 1, pp. 12--27, 2008.



\bibitem{Ferber1999} 
J. Ferber and G. Weiss, {\it{Multi-agent systems: an introduction to distributed artificial intelligence}}. Reading: Addison-wesley, 1999.


\bibitem{KrausAI1977}
S. Kraus, ``Negotiation and cooperation in multi-agent environments,'' \textit{Artif. Intell.}, vol. 4, no. 1, vol. 94, no. 1, pp. 79--97,  1997.



\bibitem{MintzBCHAOS25} 
B. Mintz and F. Fu, ``Evolutionary multi-agent reinforcement learning in group social dilemmas,''\textit{ Chaos}, vol. 35, no. 2, pp. 023140, 2025.

\bibitem{Nguyencyb20} 
T. T. Nguyen, N. D. Nguyen, and S. Nahavandi, ``Deep reinforcement learning for multiagent systems: A review of challenges, solutions, and applications,''\textit{IEEE Trans. Cybern.}, vol. 50, no. 9, pp. 3826–3839, 2020.




\bibitem{VincentCUP} 
T. L. Vincent and J. S. Brown, {\it{Evolutionary Game Theory, Natural Selection, and Darwinian Dynamics}},   Cambridge University Press, 2005.


\bibitem{Tangcyb14} 
C. Tang, A. Li, and X. Li, ``When reputation enforces evolutionary cooperation in unreliable MANETs,''\textit{IEEE Trans. Cybern.}, vol. 45, no. 10, pp. 2190--2201, 2014.




\bibitem{AxelrodBBNY} 
R. M. Axelrod, {\it{The Evolution of Cooperation}},   Basic Books: New York, 2006.



\bibitem{RandTCS} 
D. G. Rand and  M. A. Nowak, ``Human cooperation,''\textit{ Trends Cogn. Sci.}, vol. 17, no. 8, pp. 413--425, 2013.



\bibitem{Tangtnse24} 
C. Tang, B. Yang, Y. Zhang, F. Lin, and G. Chen, ``Cooperation and optimization of multi-pool mining game with zero determinant alliance,''\textit{ IEEE Trans. Netw. Sci. Eng.}, vol. 11, no. 5, pp. 4965--4978, 2024.







\bibitem{RiehlARC18} 
J. Riehl, P. Ramazi, and M. Cao, ``A survey on the analysis and control
of evolutionary matrix games,''\textit{ Annu. Rev. Control}, vol. 45, pp. 87--106, 2018. 


\bibitem{LiTCNS18} 
C. Li, F. He, T. Liu, and D. Cheng, ``Verification and dynamics of group-based potential games,''\textit{ IEEE Trans. Control. Netw. Syst.}, vol. 6, no. 1, pp. 215--224, 2018.

\bibitem{Tancyb25} 
S. Tan, Y. Wang, Y. Chen, and Z. Wang, ``Evolutionary dynamics of collective behavior selection and drift: Flocking, collapse, and oscillation,''\textit{IEEE Trans. Cybern.}, vol. 47, no. 7, pp. 1694--1705, 2016.


\bibitem{Shicyb25} 
J. Shi, C. Chu, G. Fan, D. Hu, J. Liu, Z. Wang, and S. Hu, ``Payoff control in multichannel games: influencing opponent learning evolution,''\textit{IEEE Trans. Cybern.}, vol. 55, no. 2, pp. 776--785, 2025.



\bibitem{GlaubitzPNAS24} 
A. Glaubitz and F. Fu, ``Social dilemma of nonpharmaceutical interventions: Determinants of dynamic compliance and behavioral shifts,''\textit{ Proc. Nat. Acad. Sci.}, vol. 121, no. 50, pp. e2407308121, 2024.



\bibitem{Sunhb2022} 
Q. Su, A. McAvoy, Y. Mori, and J. B. Plotkin, ``Evolution of prosocial behaviours in multilayer populations,''\textit{ Nat. Hum. Behav.}, vol. 6, no. 3, pp. 338--348, 2022.




\bibitem{Zhangtnse25} 
Y. Zhang, J. Wang, G. Wen, J. Guan, S. Zhou, G. Chen, K.  Chatterjee, and M. Perc, ``Limitation of time promotes cooperation in structured collaboration systems,''\textit{ IEEE Trans. Netw. Sci. Eng.}, vol. 12, no. 1, pp. 4--12, 2025.



\bibitem{ShiJtnse24}                                        J. Shi, C. Liu, and J. Liu, ``Hypergraph-based model for modelling multi-agent Q-learning dynamics in public goods games,''\textit{ IEEE Trans. Netw. Sci. Eng.}, vol. 11, no. 6, pp. 6169--6179, 2024.





\bibitem{Dingcyb21} 
X. Ding, H. Li, J. Lu, and S. Wang, ``Optimal strategy estimation of random evolutionary Boolean games,''\textit{IEEE Trans. Cybern.}, vol. 52, no. 8, pp. 7899--7905, 2021.




\bibitem{HilbePNAS13} 
C. Hilbe, M. A. Nowak, and K. Sigmund, ``Evolution of extortion in
iterated prisoner’s dilemma games,''\textit{ Proc. Nat. Acad. Sci. USA}, vol. 110, no. 17, pp. 6913--6918, 2013.


\bibitem{Sunpnas2022} 
Q. Su, B. Allen, and J. B. Plotkin, ``Evolution of cooperation with asymmetric social interactions,''\textit{ Proc. Natl. Acad. Sci. USA}, vol. 119, no. 1, pp. e2113468118, 2022.


\bibitem{Smith1982} 
J. M. Smith, ``Evolution and the Theory of Games,'' {\it{ Cambridge University Press}}, 1982.


\bibitem{Zhutac2022} 
Y. Zhu, C. Xia, and Z. Chen, ``Nash equilibrium in iterated multiplayer games under asynchronous best-response dynamics,''\textit{ IEEE Trans. Automat. Contr.}, vol. 69, no. 9, pp.  5798--5805, 2023.



\bibitem{LiNC2020} 
A. Li, L. Zhou, Q. Su, S. P. Cornelius,  Y. Y. Liu,  L. Wang, and S. A. Levin, ``Evolution of cooperation on temporal networks,''\textit{ Nat. Commun.}, vol. 11, no. 1, pp. 2259, 2020.


\bibitem{Liucyb24} 
A. Liu, L. Wang, G. Chen, and X. Guan, ``Heterogeneously networked evolutionary games with intergroup conflicts,''\textit{IEEE Trans. Cybern.}, vol. 54, no. 10, pp. 5684--5695, 2024.



\bibitem{Gintis2009} 
H. Gintis, \emph{Game Theory Evolving},  Princeton University Press, 2009.

\bibitem{HanINTER15} 
T. A. Han, L. M. Pereira, and T. Lenaerts, ``Avoiding or restricting defectors in public goods games?''\textit{ J. R. Soc. Interface}, vol. 12, no. 103, pp. 20141203, 2015.


\bibitem{Sasaki2012} 
T. Sasaki, {\AA}. Br\"{a}nnstr\"{o}m, U. Dieckmann, and K. Sigmund, ``The take-it-or-leave-it option allows small penalties to overcome social dilemmas,''\textit{ Proc. Natl. Acad. Sci. USA}, vol. 109, no. 4, pp. 1165--1169, 2012.

\bibitem{SunTNSE2023} 
Z. Sun, X. Chen, and A. Szolnoki, ``State-dependent optimal incentive allocation protocols for cooperation in public goods games on regular networks,''\textit{ IEEE Trans. Network Sci. Eng.}, vol. 10, no. 6, pp. 3975--3988, 2023.


\bibitem{Riehl12018} 
J. R. Riehl, P. Ramazi, and  M. Cao, ``Incentive-based control of asynchronous best-response dynamics on binary decision networks,''\textit{ IEEE Trans. Control. Netw. Syst.}, vol. 6, no. 2, pp. 727--736, 2018.

\bibitem{Vasconcelos2013} 
V. V. Vasconcelos, F. C. Santos, and  J. M. Pacheco, ``A bottom-up institutional approach to cooperative governance of risky commons,''\textit{ Nat. Clim. Change}, vol. 3, no. 9, pp. 797--801, 2013.


\bibitem{Paarporn2018} 
K. Paarporn and C. Eksin, \emph{Incentive control in network anti-coordination games with binary types}, in 2018 52nd Asilomar Conference on Signals, Systems, and Computers, IEEE, pp. 316--320, 2018. 


\bibitem{Fang2019PRSA} 
Y. Fang, T. P. Benko, M. Perc, H. Xu, and Q. Tan, ``Synergistic third-party rewarding and punishment in the public goods game,''\textit{  Proc. R. Soc. A}, vol. 475, no. 2227, pp. 20190349, 2019.

\bibitem{WangCNSNS2019} 
S. Wang, X. Chen, and A. Szolnoki, ``Exploring optimal institutional incentives for public cooperation,''\textit{ Commun. Nonlinear Sci. Numer. Simul.}, vol. 79, pp. 104914, 2019.

\bibitem{Wang2022JRCS} 
S. Wang, X. Chen, Z. Xiao, A. Szolnoki, and V. V. Vasconcelos, ``Optimization of institutional incentives for promoting cooperation in structured populations,''\textit{ J. R. Soc. Interface}, vol. 20, no. 199, pp. 20220653, 2023.


\bibitem{Wang2025TAC} 
S. Wang, M. Cao,  and X. Chen, ``Optimally combined incentive for cooperation among interacting agents in population games,''\textit{IEEE Trans. Automat. Contr.}, vol. 70(7), pp. 4562-4577, 2025.



\bibitem{WangG2024PNAS} 
	G. Wang, Q. Su, L. Wang, and B. Plotkin, ``The evolution of social behaviors and risk preferences in settings with uncertainty,''\textit{Proc. Natl. Acad. Sci. USA}, vol. 121(30), pp. e2406993121, 2024.

\bibitem{FujimotoY2023PNAS} 
	Y. Fujimoto and H. Ohtsuki, ``Evolutionary stability of cooperation in indirect reciprocity under noisy and private assessment,''\textit{Proc. Natl. Acad. Sci. USA}, vol. 120(20), pp. e2300544120, 2023.


\bibitem{WangX2023NC} 
	X. Wang, L. Zhou, A. McAvoy, and et al, ``Imitation dynamics on networks with incomplete information,''\textit{Nat. Commun.}, vol. 14(1), pp. 7453, 2023.

\bibitem{Schmid2023NC} 
	L. Schmid, F. Ekbatani, C. Hilbe, and et al, ``Quantitative assessment can stabilize indirect reciprocity under imperfect information,''\textit{Nat. Commun.}, vol. 14(1), pp. 2086, 2023.

\bibitem{Zhang2024eLifE} 
	Z. Zhang, H. Wang, T. Zhang, and et al, ``Perceptual error based on Bayesian cue combination drives implicit motor adaptation,''\textit{eLife}, vol. 13, pp. RP94608, 2024.


\bibitem{Barfuss2022PRE} 
	W. Barfuss  and R. P. Mann, ``Modeling the effects of environmental and perceptual uncertainty using deterministic reinforcement learning dynamics with partial observability,''\textit{Phys. Rev. E}, vol. 105(3), pp. 034409, 2022.

\bibitem{Lupre2024} 
Z. Lu, S. Hua, L. Wang, and L. Liu, ``Hybrid reward-punishment in feedback-evolving game for common resource governance,''\textit{ Phys. Rev. E}, vol. 110, pp. 034301, 2024.

\bibitem{Schmalensee2017} 
	R. Schmalensee and R. N. Stavins, ``Lessons learned from three decades of experience with cap and trade,''\textit{Rev. Environ. Econ. Polic.}, vol. 11, no. 1, pp. 59--79, 2017. 




\bibitem{Duong2021} 
M. H. Duong  and T. A. Han, ``Cost efficiency of institutional incentives for promoting cooperation in finite populations,''\textit{ Proc. R. Soc. A}, vol. 477, no. 2254, pp. 20210568, 2021.


\bibitem{Duong24BMB} 
M. H. Duong, C. M. Durbac, and T. A. Han, ``Cost optimisation of individual-based institutional reward incentives for promoting cooperation in finite populations,''\textit{ Bull. Math. Biol.}, vol. 86, no. 9, pp. 115, 2024.




\bibitem{ChenInterface2015} 
X. Chen, T. Sasaki, {\AA}. Br\"{a}nnstr\"{o}m, and  U. Dieckmann, ``First carrot, then stick: how the adaptive hybridization of incentives promotes cooperation,''\textit{ J. R. Soc. Interface}, vol. 12, no. 102, pp. 20140935, 2015.



\bibitem{Alventosa2021} 
A. Alventosa, A. Antonioni, and P. Hern\,{a}ndez, ``Pool punishment in public goods games: How do sanctioners incentives affect us?''\textit{J. Econ. Behav. Organ.}, vol. 185, pp. 513--537, 2021.


\bibitem{Camerer1992} 
C. Camerer, ``The rationality of prices and volume in experimental markets,''\textit{ Organ. Behav. Hum. Dec.}, vol. 51, no. 2, pp. 237--272, 1992.

\bibitem{Patel2002} 
V. L. Patel, D. R. Kaufman, and J. F. Arocha, ``Emerging paradigms of cognition in medical decision-making,''\textit{ J. Biomed. Inform.}, vol. 35, no. 1, pp. 52--75, 2002.



\bibitem{Evans2005} 
L. C. Evans, ``An Introduction to Mathematical Optimal Control Theory,''  {\it{ University of California Press}}, 2005.


\bibitem{Geering2007} 
H. P. Geering, ``Optimal Control with Engineering Applications,'' {\it{ Springer}}, 2007.


\bibitem{Schuster1983} 
P. Schuster and K. Sigmund, ``Replicator dynamics,''\textit{ Journal of Theoretical Biology}, vol. 100, no. 3, pp. 533--538, 1983.

\bibitem{Hofbauer1998} 
J. Hofbauer and K. Sigmund, ``Evolutionary Games and Population Dynamics,'' {\it{ Cambridge University Press}}, 1998.

\bibitem{Szabo1998} 
G. Szab\'{o} and C. T\H{o}ke, ``Evolutionary prisoner's dilemma game on a square lattice,''\textit{Phys. Rev. E}, vol. 58, no. 1, pp. 69, 1998.

\bibitem{Traulsen2005} 
A. Traulsen, J. C. Claussen, and C. Hauert, ``Coevolutionary dynamics: from finite to infinite populations,''\textit{ Phys. Rev. Lett.}, vol. 95, no. 23, pp. 238701, 2005. 

\bibitem{Traulsen2009} 
A. Traulsen and C. Hauert, ``Stochastic evolutionary game dynamics,''\textit{Rev. Nonlin. Dyn. Complex}, vol. 2, pp. 25--61, 2009.

\bibitem{Chen2018ICML} 
Z. Chen, V. Badrinarayanan, C. Y. Lee, and A. Rabinovich, ``Gradnorm: Gradient normalization for adaptive loss balancing in deep multitask networks,'' \textit{in International Conference on Machine Learning}, pp.794--803, 2018.

\bibitem{Spitzer2025MS} 
	F. Spitzer, K. Abstiens, and S. Karmasin, ``Integrating behavioural insights in the policy process: on chances and hurdles identified by policy-makers and behavioural scientists,''\textit{Mind Soc.}, vol. 24, no. 2, pp. 621--663, 2025.


\bibitem{Ferguson2022cdc} 
	B. L. Ferguson, P. N. Brown, and J. R. Marden, ``Avoiding unintended consequences: How incentives aid information provisioning in Bayesian congestion game,''\textit{In 2022 IEEE 61st Conference on Decision and Control (CDC)}, pp. 3781--3786, 2022.


\bibitem{Mann2017PNAS} 
	R. P. Mann and D. Helbing, ``Optimal incentives for collective intelligence,''\textit{Proc. Natl. Acad. Sci. USA}, vol. 114, no. 20, pp. 5077--5082, 2017.


\bibitem{WeiHCYB26} 
	H. Wei, J. Zhang, C. Zhang, and M. Cao, ``Indirect reciprocity enhances collective cooperation on weighted networks,''\textit{ IEEE Trans. Cybern.}, vol. 56, no. 2, pp. 1141--1152, 2026.


\bibitem{ZhuYHCYB26} 
	Y. Zhu, Z. Zhang, C. Xia, X. Li, and Z. Chen, ``Finite Strategy Switches of Coordinating and Anti-Coordinating Games on Weighted Networks,''\textit{ IEEE Trans. Cybern.}, vol. 56, no. 1, pp. 446--459, 2026. 

\bibitem{LiuAYHCYB26} 
	A. Liu, L. Wang, G. Chen, and X. Guan, ``Heterogeneously networked evolutionary games with intergroup conflicts,''\textit{ IEEE Trans. Cybern.}, vol. 56, no. 10, pp. 5684--5695, 2024. 


\bibitem{BarabSCi1999} 
	A.~L.~Barab\'{a}si and R.~Albert, ``Emergence of scaling in random networks,''\textit{Science}, vol. 286, no. 5439, pp. 509--512, 1999.

\bibitem{WattsNat1998} 
	D.~J.~Watts and  S.~H.~Strogatz, ``Collective dynamics of `small-world' networks,''\textit{Nature}, vol. 393, no. 6684, pp. 440--442, 1998.

\bibitem{WangTNNLS2025} 
	Q. Wang, X. Chen, N. He, and A. Szolnokiz, ``Evolutionary dynamics of population games with an aspiration-based learning rule,''\textit{IEEE Trans. Neural Netw. Learn. Syst.}, vol. 36, no. 5, pp. 8387--8400, 2025.



\bibitem{Ohtsuk2006Nature} 
	H. Ohtsuki, C. Hauert,  E. Lieberman, and M. A. Nowak, ``A simple rule for the evolution of cooperation on graphs and social networks,''\textit{Nature}, vol. 441, no. 7092, pp. 502--505, 2006.




\bibitem{VanKampen1992} 
N. G. Van Kampen, {\it{Stochastic Processes in Physics and Chemistry}},   Elsevier, 1992.


\bibitem{Ohtsuk2006Nature} 
H. Ohtsuki, C. Hauert,  E. Lieberman, and M. A. Nowak, ``A simple rule for the evolution of cooperation on graphs and social networks,''\textit{Nature}, vol. 441, no. 7092, pp. 502--505, 2006.





\end{thebibliography}
\end{document}